\documentclass{amsart}
\usepackage{amsmath,amsthm,amsfonts,amssymb}

\numberwithin{equation}{section}
\theoremstyle{plain}

\newtheorem{theorem}{Theorem}

\newtheorem{corollary}[theorem]{Corollary}

\newtheorem{lemma}[theorem]{Lemma}

\newtheorem{proposition}[theorem]{Proposition}

\newtheorem{remark}[theorem]{Remark}

\begin{document}

\title[Equality of averaged and quenched large deviations for RWRE]{Equality of averaged and quenched large deviations \\for random walks in random environments \\in dimensions four and higher}
\author{Atilla Yilmaz}
\address{Faculty of Mathematics\\
Weizmann Institute of Science}
\curraddr{Department of Mathematics\\
University of California, Berkeley}
\email{atilla@math.berkeley.edu}
\date{February 25, 2009. Revised on December 30, 2009}


\begin{abstract}

We consider large deviations for nearest-neighbor random walk in a uniformly elliptic i.i.d.\ environment. It is easy to see that the quenched and the averaged rate functions are not identically equal. When the dimension is at least four and Sznitman's transience condition (T) is satisfied, we prove that these rate functions are finite and equal on a closed set whose interior contains every nonzero velocity at which the rate functions vanish.

\end{abstract}

\maketitle

\section{Introduction}

\subsection{The model}

Let $U:=\{\pm e_i\}_{i=1}^d$ where $(e_i)_{i=1}^d$ denotes the canonical basis for the $d$-dimensional integer lattice $\mathbb{Z}^d$ with $d\geq1$. Consider a discrete time Markov chain on $\mathbb{Z}^d$ with nearest-neighbor steps, i.e., with steps in $U$. For every $x\in\mathbb{Z}^d$ and $z\in U$, the transition probability from $x$ to $x+z$ is denoted by $\pi(x,x+z)$, and the transition vector $\omega_x:=(\pi(x,x+z))_{z\in U}$ is referred to as the \textit{environment} at $x$. If the environment $\omega:=(\omega_x)_{x\in\mathbb{Z}^d}$ is sampled from a probability space $(\Omega,\mathcal{B},\mathbb{P})$, then this process is called \textit{random walk in a random environment} (RWRE). Here, $\mathcal{B}$ is the Borel $\sigma$-algebra corresponding to the product topology.

For every $y\in\mathbb{Z}^d$, define the shift $T_y$ on $\Omega$ by $\left(T_y\omega\right)_x:=\omega_{x+y}$. Assume that $\mathbb{P}$ is stationary and ergodic under $\left(T_z\right)_{z\in U}$ and 
\begin{equation}\label{ellipticity}
\mbox{there exists a $\delta>0$ such that $\mathbb{P}\{\pi(0,z)\geq\delta\}=1$ for every $z\in U$. (\textit{Uniform ellipticity}.)}
\end{equation}

For every $x\in\mathbb{Z}^d$ and $\omega\in\Omega$, the Markov chain with environment $\omega$ induces a probability measure $P_x^\omega$ on the space of paths starting at $x$. Statements about $P_x^\omega$ that hold for $\mathbb{P}$-a.e.\ $\omega$ are referred to as \textit{quenched}. Statements about the semi-direct product $P_x:=\mathbb{P}\times P_x^\omega$ are referred to as \textit{averaged} (or \textit{annealed}). Expectations under $\mathbb{P}, P_x^\omega$ and $P_x$ are denoted by $\mathbb{E}, E_x^\omega$ and $E_x$, respectively.

See \cite{Zeitouni} for a survey of results on RWRE.

\subsection{Regeneration times}

Let $\left(X_n\right)_{n\geq0}$ denote the path of a particle taking a RWRE. Consider a unit vector $\hat{u}\in\mathcal{S}^{d-1}$.
Define a sequence $\left(\tau_m\right)_{m\geq0}=\left(\tau_m(\hat{u})\right)_{m\geq0}$ of random times, which are referred to as \textit{regeneration times} (relative to $\hat{u}$), by $\tau_o:=0$ and
\begin{align}
\tau_{m}&:=\inf\left\{j>\tau_{m-1}:\langle X_i,\hat{u}\rangle<\langle X_j,\hat{u}\rangle\leq\langle X_k,\hat{u}\rangle\mbox{ for all }i,k\mbox{ with }i<j<k\right\}\label{regenerationtimes_m}
\end{align}
for every $m\geq1$.
If the walk is directionally transient relative to $\hat{u}$, i.e., if
\begin{equation}\label{transience}
P_o\left(\lim_{n\to\infty}\langle X_n,\hat{u}\rangle=\infty\right)=1,
\end{equation}
then $P_o\left(\tau_m<\infty\right)=1$ for every $m\geq1$.
As noted in \cite{SznitmanZerner}, the significance of $\left(\tau_m\right)_{m\geq1}$ is due to the fact that $$\left(X_{\tau_m+1}-X_{\tau_m},X_{\tau_m+2}-X_{\tau_m},\ldots,X_{\tau_{m+1}}-X_{\tau_m}\right)_{m\geq1}$$ is an i.i.d.\ sequence under $P_o$ when
\begin{equation}\label{i.i.d.}
\omega=(\omega_x)_{x\in\mathbb{Z}^d}\mbox{ is an i.i.d.\ collection}.
\end{equation}

The walk is said to satisfy Sznitman's transience condition (\textbf{T},$\hat{u}$) if (\ref{transience}) holds and
\begin{equation}\label{moment}
E_o\left[\sup_{1\leq i\leq\tau_1(\hat{u})}\exp\left\{\kappa_1\left|X_i\right|\right\}\right]<\infty\mbox{ for some }\kappa_1>0.
\end{equation}
When $d\geq2$, Sznitman \cite{SznitmanT} proves that (\ref{ellipticity}), (\ref{i.i.d.}) and (\textbf{T},$\hat{u}$) imply a \textit{ballistic} law of large numbers (LLN), an averaged central limit theorem and certain large deviation estimates. Denote the LLN velocity by $\xi_o\neq0$.

As stated below in Lemma \ref{madaiki}, (\textbf{T},$\hat{u}$) is satisfied as soon as the walk is \textit{non-nestling} relative to $\hat{u}$, i.e., when
\begin{equation}\label{nonmumu}
\mathrm{ess}\inf_{\mathbb{P}}\sum_{z\in U}\pi(0,z)\langle z,\hat{u}\rangle>0.
\end{equation}
The walk is said to be non-nestling if it is non-nestling relative to some unit vector. Otherwise, it is referred to as \textit{nestling}. In the latter case, the convex hull of the support of the law of $\sum_{z}\pi(0,z)z$ contains the origin.

\subsection{Previous results on large deviations for RWRE}

Recall that a sequence $\left(Q_n\right)_{n\geq1}$ of probability measures on a topological space $\mathbb{X}$ satisfies the \textit{large deviation principle} (LDP) with rate function $I:\mathbb{X}\to\mathbb{R}$ if $I$ is non-negative, lower semicontinuous, and for any measurable set $G$, $$-\inf_{x\in G^o}I(x)\leq\liminf_{n\to\infty}\frac{1}{n}\log Q_n(G)\leq\limsup_{n\to\infty}\frac{1}{n}\log Q_n(G)\leq-\inf_{x\in\bar{G}}I(x).$$ Here, $G^o$ is the interior of $G$, and $\bar{G}$ its closure. See \cite{DemboZeitouni} for general background regarding large deviations.

\begin{theorem}[Quenched LDP]\label{qLDPgeneric}
For $\mathbb{P}$-a.e.\ $\omega$, $\left(P_o^\omega\left(\frac{X_n}{n}\in\cdot\,\right)\right)_{n\geq1}$ satisfies the LDP with a deterministic and convex rate function $I_q$.
\end{theorem}

When $d=1$, Greven and den Hollander \cite{GdH} prove Theorem \ref{qLDPgeneric} for walks in i.i.d.\ environments. They provide a formula for $I_q$ and show that its graph typically has flat pieces. Comets, Gantert and Zeitouni \cite{CGZ} generalize the results in \cite{GdH} to stationary and ergodic environments.

When $d\geq1$, Zerner \cite{Zerner} proves Theorem \ref{qLDPgeneric} for nestling walks in i.i.d.\ environments. Varadhan \cite{Raghu} drops the nestling assumption and generalizes Theorem \ref{qLDPgeneric} to stationary and ergodic environments. Since both of these works rely on the subadditive ergodic theorem, they do not lead to any formulae for the rate function. Rosenbluth \cite{jeffrey} gives an alternative proof of Theorem \ref{qLDPgeneric} in the case of stationary and ergodic environments. He provides a variational formula for the rate function $I_q$. In \cite{YilmazQuenched}, we prove a quenched LDP for the pair empirical measure of the so-called \textit{environment Markov chain} $\left(T_{X_n}\omega\right)_{n\geq0}$. This implies Rosenbluth's result by an appropriate contraction.

In their aforementioned paper concerning RWRE on $\mathbb{Z}$, Comets et al.\ \cite{CGZ} prove also

\begin{theorem}[Averaged LDP]\label{aLDPgeneric}
$\left(P_o\left(\frac{X_n}{n}\in\cdot\,\right)\right)_{n\geq1}$ satisfies the LDP with a convex rate function $I_a$.
\end{theorem}

\noindent They establish this result for a class of environments including the i.i.d.\ case, and obtain the following variational formula for $I_a$:
\begin{equation}\label{montreal}
I_a(\xi)=\inf_{\mathbb{Q}}\left\{I_q^\mathbb{Q}(\xi) + |\xi|h\left(\mathbb{Q}\left|\mathbb{P}\right.\right)\right\}.
\end{equation}Here, the infimum is over all stationary and ergodic probability measures on $\Omega$, $I_q^\mathbb{Q}(\cdot)$ denotes the rate function for the quenched LDP when the environment measure is $\mathbb{Q}$, and $h\left(\cdot\left|\cdot\right.\right)$ is specific relative entropy. Similar to the quenched picture, the graph of $I_a$ is shown to typically have flat pieces.

Varadhan \cite{Raghu} considers walks in i.i.d.\ environments, and proves Theorem \ref{aLDPgeneric} for any $d\geq1$. He gives a variational formula for $I_a$. (His formula does not resemble (\ref{montreal}) in any way.) Rassoul-Agha \cite{FirasLDP} generalizes Varadhan's result to a class of mixing environments.

Let $\mathcal{N}_q:=\left\{\xi\in\mathbb{R}^d:I_q(\xi)=0\right\}$ and $\mathcal{N}_a:=\left\{\xi\in\mathbb{R}^d:I_a(\xi)=0\right\}$ denote the zero-sets of $I_q$ and $I_a$. The following theorem summarizes the previous results regarding the qualitative properties of the quenched and the averaged rate functions when $d\geq2$.
\begin{theorem}\label{previousqual}
Assume $d\geq2$, (\ref{ellipticity}) and (\ref{i.i.d.}).
\begin{itemize}
\item[(a)] $I_q$ and $I_a$ are convex, $I_q(0)=I_a(0)$ and $\mathcal{N}_q=\mathcal{N}_a$, cf.\ \cite{Raghu}.
\item[(b)] If the walk is non-nestling, then 
\begin{itemize}
\item[(i)] $\mathcal{N}_a$ consists of the true velocity $\xi_o$, cf.\ \cite{Raghu}, and
\item[(ii)] $I_a$ is strictly convex and analytic on an open set $\mathcal{A}_a$ containing $\xi_o$, cf.\ \cite{Jon,YilmazAveraged}.
\end{itemize}
\item[(c)] If the walk is nestling, then $\mathcal{N}_a$ is a line segment containing the origin, cf.\ \cite{Raghu}.
\item[(d)] If (\textbf{T},$\hat{u}$) is satisfied for some $\hat{u}\in\mathcal{S}^{d-1}$ in the latter case, then 
\begin{itemize}
\item[(i)] the origin is an endpoint of $\mathcal{N}_a$, cf.\ \cite{SznitmanT},
\item[(ii)] $I_a$ is strictly convex and analytic on an open set $\mathcal{A}_a^+$, cf.\ \cite{YilmazAveraged},
\item[(iii)] there exists a $(d-1)$-dimensional smooth surface patch $\mathcal{A}_a^b$ such that $\xi_o\in\mathcal{A}_a^b\subset\partial\mathcal{A}_a^+$, cf.\ \cite{YilmazAveraged},
\item[(iv)] the unit vector $\eta_o$ normal to $\mathcal{A}_a^b$ (and pointing in $\mathcal{A}_a^+$) at $\xi_o$ satisfies $\langle\eta_o,\xi_o\rangle>0$, cf.\ \cite{YilmazAveraged}, and
\item[(v)] $I_a(t\xi)=tI_a(\xi)$ for every $\xi\in\mathcal{A}_a^b$ and $t\in[0,1]$, cf.\ \cite{Jon}.
\end{itemize}
\end{itemize}
\end{theorem}

\subsection{The main result}

Assume (\ref{ellipticity}) and (\ref{i.i.d.}). It is clear that
$$\mathcal{D}:=\left\{(\xi_1,\ldots,\xi_d)\in\mathbb{R}^d:|\xi_1|+\cdots+|\xi_d|\leq1\right\}=\left\{\xi\in\mathbb{R}^d:I_a(\xi)<\infty\right\}=\left\{\xi\in\mathbb{R}^d:I_q(\xi)\leq-\log\delta\right\}.$$
For any $\xi\in\mathbb{R}^d$, $I_a(\xi)\leq I_q(\xi)$ by Jensen's inequality and Fatou's lemma.

\begin{proposition}\label{sevgilihikmet}
If the support of $\mathbb{P}$ is not a singleton, then $I_a<I_q$ at some interior points of $\mathcal{D}$.
\end{proposition}

\begin{proof}
If the support of $\mathbb{P}$ is not a singleton, then $\mathbb{P}\left\{\pi(0,z)=\mathbb{E}\{\pi(0,z)\}\right\}<1$ for some $z\in U$, and
\begin{equation}\label{sakkino}\mathbb{E}\{\log\pi(0,z)\}<\log\mathbb{E}\{\pi(0,z)\}
\end{equation} by Jensen's inequality. For every $n\geq1$, the event $\{X_n=nz\}$ consists of a single path marching in the $z$-direction. In particular, this path never visits the same point more than once. Therefore,
\begin{equation}\label{kabusgibiydi}
\log\mathbb{E}\{\pi(0,z)\}=\lim_{n\to\infty}\frac{1}{n}\log P_o(X_n=nz)\leq-I_a(z).
\end{equation}
On the other hand, for every $\epsilon>0$,
\begin{equation}\label{gecensenezordu}
-I_q(z)\leq\liminf_{n\to\infty}\frac1n\log P_o^\omega(\langle X_n,z\rangle>n(1-\epsilon))\leq(1-\epsilon)\mathbb{E}\{\log\pi(0,z)\}+O(\epsilon).
\end{equation}
Explanation: For every $n\geq1$, the number of paths constituting the event $\{\langle X_n,z\rangle>n(1-\epsilon)\}$ is $\mathrm{e}^{nO(\epsilon)}$. The probability of each such path is bounded from above by the product of the probabilities of its jumps in the $z$-direction taking place at distinct points. Since there are at least $n(1-\epsilon)$ such jumps, (\ref{gecensenezordu}) follows from Jensen's inequality and the LLN for i.i.d.\ random variables.

Putting (\ref{sakkino}), (\ref{kabusgibiydi}) and (\ref{gecensenezordu}) together, we conclude that $I_a(z)<I_q(z)$. Since the rate functions are convex and lower semicontinuous, they are in fact continuous on $\mathcal{D}$, cf.\ Theorem 10.2 of \cite{Rockafellar}. This implies the desired result.
\end{proof}

The following theorem is the main result of this paper.
\begin{theorem}\label{QequalsA}
Assume $d\geq4$, (\ref{ellipticity}), (\ref{i.i.d.}) and (\textbf{T},$\hat{u}$) for some $\hat{u}\in\mathcal{S}^{d-1}$.
\begin{itemize}
\item[(a)] If the walk is non-nestling, then $I_q=I_a$ on an open set $\mathcal{A}_{eq}$ containing $\xi_o$.
\item[(b)] If the walk is nestling, then
\begin{itemize}
\item[(i)] $I_q=I_a$ on an open set $\mathcal{A}_{eq}^+$,
\item[(ii)] there exists a $(d-1)$-dimensional smooth surface patch $\mathcal{A}_{eq}^b$ such that $\xi_o\in\mathcal{A}_{eq}^b\subset\partial\mathcal{A}_{eq}^+$,
\item[(iii)] the unit vector $\eta_o$ normal to $\mathcal{A}_{eq}^b$ (and pointing in $\mathcal{A}_{eq}^+$) at $\xi_o$ satisfies $\langle\eta_o,\xi_o\rangle>0$, and
\item[(iv)] $I_q(t\xi)=tI_q(\xi)=tI_a(\xi)=I_a(t\xi)$ for every $\xi\in\mathcal{A}_{eq}^b$ and $t\in[0,1]$.
\end{itemize}
\end{itemize}
\end{theorem}

\noindent\textbf{Some remarks.}
\begin{itemize}
\item[1.] Since $I_q$ and $I_a$ are both continuous on $\mathcal{D}$, it is clear that $\mathcal{E}:=\left\{\xi\in\mathcal{D}:I_q(\xi)=I_a(\xi)\right\}$ is closed. Proposition \ref{sevgilihikmet} and Theorem \ref{QequalsA} imply that $\mathcal{D}\setminus\mathcal{E}$ and $\mathcal{E}$ both have nonempty interiors.
\item[2.] Assuming $d=1$, (\ref{ellipticity}) and (\ref{i.i.d.}), Comets et al.\ \cite{CGZ} use (\ref{montreal}) to show that $I_q(\xi)=I_a(\xi)$ if and only if $\xi=0$ or $I_a(\xi)=0$. In particular, Theorem \ref{QequalsA} cannot be generalized to $d\geq1$. Whether it can be generalized to $d\geq2$ is an open problem.
\item[3.] The analog of Theorem \ref{QequalsA} for so-called \textit{space-time} RWRE is proved in \cite{YilmazSpaceTime}.
\item[4.] Related results have been obtained for random walks in random potentials, cf.\ \cite{Flury,Nikos}, for directed polymers in random environments, cf.\ \cite{CometsShigaYoshida}, and for random walks on Galton-Watson trees, cf.\ \cite{RWREGW,lambdaGW,RWREH}.
\end{itemize}
\vspace{0.5cm}

\section{Proof of the main result}\label{turetmis}

\subsection{Outline}\label{outline}

For every $\theta\in\mathbb{R}^d$, consider the logarithmic moment generating functions $$\Lambda_q(\theta):=\lim_{n\to\infty}\frac{1}{n}\log E_o^\omega\left[\exp\{\langle\theta,X_n\rangle\}\right]\quad\mbox{and}\quad\Lambda_a(\theta):=\lim_{n\to\infty}\frac{1}{n}\log E_o\left[\exp\{\langle\theta,X_n\rangle\}\right].$$
By Varadhan's Lemma, cf.\ \cite{DemboZeitouni}, $\Lambda_q(\theta)=\sup_{\xi\in\mathbb{R}^d}\left\{\langle\theta,\xi\rangle - I_q(\xi)\right\}=I_q^*(\theta)$, the convex conjugate of $I_q$ at $\theta$. Similarly, $\Lambda_a(\theta)=I_a^*(\theta)$.

Assume $d\geq4$ and (\textbf{T},$\hat{u}$) for some $\hat{u}\in\mathcal{S}^{d-1}$. For every $n\geq0$, $\theta\in\mathbb{R}^d$ and $\omega\in\Omega$, define 
\begin{align}
H_n=H_n(\hat{u})&:=\inf\left\{i\geq0: \langle X_i,\hat{u}\rangle\geq n\right\},\quad\beta=\beta(\hat{u}):=\inf\left\{i\geq0: \langle X_i,\hat{u}\rangle<\langle X_o,\hat{u}\rangle\right\}\quad\mbox{and}\label{randomtimes}\\
g_n(\theta,\omega)&:=E_o^\omega\left[\exp\left\{\langle\theta,X_{H_n}\rangle-\Lambda_a(\theta)H_n\right\},H_n=\tau_k\ \mbox{for some }k\geq1, \beta=\infty\right].\nonumber
\end{align} 
When $|\theta|$ is sufficiently small (and $\Lambda_a(\theta)>0$ in the nestling case), we show that $\left(g_n(\theta,\cdot)\right)_{n\geq1}$ is bounded in $L^2(\mathbb{P})$ and $\mathbb{E}\left\{g_n(\theta,\cdot)\right\}$ converges to a nonzero limit as $n\to\infty$. These two facts imply that $\Lambda_q(\theta)=\Lambda_a(\theta)$.

Section \ref{L2section} is devoted to the $L^2$ estimate regarding $\left(g_n(\theta,\cdot)\right)_{n\geq1}$ which constitutes the core of this paper. Assuming that, the equality of the logarithmic moment generating functions is established in Subsection \ref{guyaispat_bir}. Finally, convex duality is used in Subsection \ref{guyaispat_iki} to prove Theorem \ref{QequalsA} by showing that the local equality of $\Lambda_q$ and $\Lambda_a$ implies the equality of $I_q$ and $I_a$ on certain subsets of $\mathcal{D}$.

We find it more convenient to work with regeneration times relative to a $z\in U$ rather than any $\hat{u}\in\mathcal{S}^{d-1}$. In Subsection \ref{prelim}, we give some results which imply that there is no loss of generality in doing so.

\subsection{Some preliminaries regarding regenerations}\label{prelim}

Assume $d\geq2$, (\ref{ellipticity}) and (\ref{i.i.d.}). 

\begin{lemma}[Sznitman \cite{SznitmanT}]\label{madabir}
Assume (\textbf{T},$\hat{u}$) for some $\hat{u}\in\mathcal{S}^{d-1}$.
\begin{itemize}
\item[(a)] $P_o(\beta(\hat{u})=\infty)>0$, and $\tau_1(\hat{u})$ has finite $P_o$-moments of arbitrary order.
\item[(b)] The LLN holds with a limiting velocity $\xi_o$ such that $\langle\xi_o,\hat{u}\rangle>0$.
\item[(c)] (\textbf{T},$\hat{v}$) is satisfied for every $\hat{v}\in\mathcal{S}^{d-1}$ such that $\langle\xi_o,\hat{v}\rangle>0$.
\end{itemize}
\end{lemma}

\begin{lemma}[Sznitman \cite{SznitmanSlowdown}]\label{madaiki}
If the walk is non-nestling relative to some $\hat{u}\in\mathcal{S}^{d-1}$, then
$$E_o\left[\exp\left\{\kappa_2\tau_1(\hat{u})\right\}\right]<\infty$$ for some $\kappa_2>0$. In particular, (\textbf{T},$\hat{u}$) is satisfied.
\end{lemma}

\begin{lemma}\label{madauc}
If the walk is non-nestling and some $\hat{v}\in\mathcal{S}^{d-1}$ satisfies $\langle\xi_o,\hat{v}\rangle>0$, then $$E_o\left[\exp\left\{c\tau_1(\hat{v})\right\}\right]<\infty$$ for some $c>0$.
\end{lemma}

\begin{proof}
Since the walk is non-nestling, (\ref{nonmumu}) holds for some $\hat{u}\in\mathcal{S}^{d-1}$ with rational coordinates. Let $a\geq1$ be an integer such that
$a\hat{u}$ has integer coordinates. Note that $\langle x,\hat{u}\rangle>0$ if and only if $\langle x,\hat{u}\rangle\geq\frac1a$ for $x\in\mathbb{Z}^d$.
Therefore, $\left|X_{\tau_{ak+1}(\hat{u})}\right|\geq\langle X_{\tau_{ak+1}(\hat{u})},\hat{u}\rangle>k$ for every $k\geq1$.

For every $c,c'>0$ and $\hat{v}\in\mathcal{S}^{d-1}$ such that $\langle\xi_o,\hat{v}\rangle>0$,
\begin{align}
E_o\left[\exp\left\{c\tau_1(\hat{v})\right\}\right]&=\sum_{k=1}^\infty E_o\left[\exp\left\{c\tau_1(\hat{v})\right\},\sup_{1\leq i\leq\tau_1(\hat{v})}\!\!\!\!\left|X_i\right|\in(k-1,k]\right]\nonumber\\
&\leq\sum_{k=1}^\infty E_o\left[\exp\left\{c\tau_{ak+1}(\hat{u})\right\},\sup_{1\leq i\leq\tau_1(\hat{v})}\!\!\!\!\left|X_i\right|\in(k-1,k]\right]\nonumber\\
&\leq\sum_{k=1}^\infty E_o\left[\exp\left\{c\tau_{ak+1}(\hat{u})\right\}\left(\sup_{1\leq i\leq\tau_1(\hat{v})}\exp\left\{c'\left|X_i\right|\right\}\right)\right]\exp\{-c'(k-1)\}\nonumber\\
&\leq E_o\left[\sup_{1\leq i\leq\tau_1(\hat{v})}\exp\left\{2c'\left|X_i\right|\right\}\right]^{1/2}\sum_{k=1}^\infty E_o\left[\exp\left\{2c\tau_{ak+1}(\hat{u})\right\}\right]^{1/2}\exp\{-c'(k-1)\}.\label{sonuncu}
\end{align}
Note that (\textbf{T},$\hat{u}$) is satisfied by Lemma \ref{madaiki}. Since $\langle\xi_o,\hat{v}\rangle>0$, it follows from Lemma \ref{madabir} that (\textbf{T},$\hat{v}$) is satisfied as well. Therefore, (\ref{moment}) implies that the first term in (\ref{sonuncu}) is finite when $c'>0$ is small enough.

It is immediate from the renewal structure that $$E_o\left[\exp\left\{2c\tau_{ak+1}(\hat{u})\right\}\right]^{1/2}=E_o\left[\exp\left\{2c\tau_{1}(\hat{u})\right\}\right]^{1/2}E_o\left[\left.\exp\left\{2c\tau_{1}(\hat{u})\right\}\right|\beta(\hat{u})=\infty\right]^{ak/2}.$$ By Lemma \ref{madaiki}, $E_o\left[\left.\exp\left\{\kappa_2\tau_{1}(\hat{u})\right\}\right|\beta(\hat{u})=\infty\right]<\infty$ for some $\kappa_2>0$. When $c>0$ is small enough, $$E_o\left[\left.\exp\left\{2c\tau_{1}(\hat{u})\right\}\right|\beta(\hat{u})=\infty\right]^{a/2}\leq E_o\left[\left.\exp\left\{\kappa_2\tau_{1}(\hat{u})\right\}\right|\beta(\hat{u})=\infty\right]^{ac/\kappa_2}<\mathrm{e}^{c'}$$ and the summation in (\ref{sonuncu}) is finite. This implies the desired result.
\end{proof}

\begin{corollary}\label{mazaltov}
Assume (\textbf{T},$\hat{u}$) for some $\hat{u}\in\mathcal{S}^{d-1}$. Since $\xi_o\neq0$, $\langle\xi_o,z\rangle>0$ for some $z\in U$. 
\begin{itemize}
\item[(a)] $P_o(\beta(z)=\infty)>0$, and $\tau_1(z)$ has finite $P_o$-moments of arbitrary order.
\item[(b)] If the walk is non-nestling, then there exists a $\kappa_3>0$ such that $$E_o\left[\exp\left\{2\kappa_3\tau_1(z)\right\}\right]<\infty.$$
\item[(c)] If the walk is nestling, then there exists a $\kappa_3>0$ such that $$E_o\left[\sup_{1\leq i\leq\tau_1(z)}\exp\left\{\kappa_3\left|X_i\right|\right\}\right]<\infty.$$
\end{itemize}
\end{corollary}

\subsection{Equality of the logarithmic moment generating functions}\label{guyaispat_bir}

Assume $d\geq4$, (\ref{ellipticity}), (\ref{i.i.d.}) and (\textbf{T},$\hat{u}$) for some $\hat{u}\in\mathcal{S}^{d-1}$. Since $\xi_o\neq0$, $\langle\xi_o,z\rangle>0$ for some $z\in U$. Assume WLOG that $\langle\xi_o,e_1\rangle>0$. Refer to (\ref{regenerationtimes_m}) and (\ref{randomtimes}) for the definitions of $$(\tau_m)_{m\geq1}=(\tau_m(e_1))_{m\geq1},\quad(H_n)_{n\geq0}=(H_n(e_1))_{n\geq0}\quad\mbox{and}\quad\beta=\beta(e_1).$$

Fix $\kappa_3$ as in Corollary \ref{mazaltov}. For every $\kappa\in(0,\kappa_3]$, define
\begin{equation}\label{nebiliyim}
\mathcal{C}_a(\kappa):=\left\{\begin{array}{ll}
\left\{\theta\in\mathbb{R}^d:|\theta|<\kappa\right\}&\mbox{if the walk is non-nestling,}\\
\left\{\theta\in\mathbb{R}^d:|\theta|<\kappa\,, \Lambda_a(\theta)>0\right\}&\mbox{if the walk is nestling.}
\end{array}\right.
\end{equation}
By Jensen's inequality, 
\begin{equation}\label{noldush}
\langle\theta,\xi_o\rangle=\lim_{n\to\infty}\frac{1}{n} E_o\left[\langle\theta,X_n\rangle\right]\leq\lim_{n\to\infty}\frac{1}{n}\log E_o\left[\exp\{\langle\theta,X_n\rangle\}\right]=\Lambda_a(\theta)\leq\lim_{n\to\infty}\frac{1}{n}\log E_o\left[\mathrm{e}^{|\theta|n}\right]=|\theta|.
\end{equation} In the nestling case, $\left\{\theta\in\mathbb{R}^d:|\theta|<\kappa,\,\langle\theta,\xi_o\rangle>0\right\}\subset\mathcal{C}_a(\kappa)$ by (\ref{noldush}). Hence, $\mathcal{C}_a(\kappa)$ is a non-empty open set both for nestling and non-nestling walks.

\begin{lemma}\label{phillysh_one}
$E_o\left[\left.\exp\{\langle\theta,X_{\tau_1}\rangle-\Lambda_a(\theta)\tau_1\}\right|\beta=\infty\right]=1$ for every $\theta\in\mathcal{C}_a(\kappa_3)$.
\end{lemma}

\begin{proof}
This is Lemma 12 of \cite{YilmazAveraged}.
\end{proof}

For every $\theta\in\mathcal{C}_a(\kappa_3)$ and $y\in\mathbb{Z}^d$, let
\begin{equation}\label{immydit}
q^\theta(y):=E_o\left[\left.\exp\{\langle\theta,X_{\tau_1}\rangle-\Lambda_a(\theta)\tau_1\}, X_{\tau_1}=y\,\right|\beta=\infty\right].
\end{equation}
Since $\sum_{y\in\mathbb{Z}^d}q^\theta(y)=1$ by Lemma \ref{phillysh_one}, $\left(q^\theta(y)\right)_{y\in\mathbb{Z}^d}$ defines a random walk $(Y_k)_{k\geq0}$ on $\mathbb{Z}^d$. For every $x\in\mathbb{Z}^d$, this walk induces a probability measure $\hat{P}_x^\theta$ on paths starting at $x$. As usual, $\hat{E}_x^\theta$ denotes the corresponding expectation. It follows from Corollary \ref{mazaltov} and H\"older's inequality that
\begin{equation}\label{onedaylate}
\hat{E}_o^\theta\left[|Y_1|^m\right]<\infty\mbox{ for every }m\geq1.
\end{equation}

For every $n\geq1$, $\theta\in\mathcal{C}_a(\kappa_3)$ and $\omega\in\Omega$, recall from Subsection \ref{outline} that
$$g_n(\theta,\omega):=E_o^\omega\left[\exp\left\{\langle\theta,X_{H_n}\rangle-\Lambda_a(\theta)H_n\right\},H_n=\tau_k\ \mbox{for some }k\geq1, \beta=\infty\right].$$
\begin{lemma}\label{L1lemma}
For every $\theta\in\mathcal{C}_{a}(\kappa_3)$,
$$\lim_{n\to\infty}\mathbb{E}\left\{g_n(\theta,\cdot)\right\}={P_o(\beta=\infty)}/{\hat{E}_o^\theta\left[\langle Y_1,e_1\rangle\right]}>0.$$
\end{lemma}

\begin{proof}
For every $n\geq1$ and $\theta\in\mathcal{C}_{a}(\kappa_3)$,
\begin{align*}
\mathbb{E}\left\{g_n(\theta,\cdot)\right\}&=E_o\left[\exp\left\{\langle\theta,X_{H_n}\rangle-\Lambda_a(\theta)H_n\right\},H_n=\tau_k\ \mbox{for some }k\geq1,\beta=\infty\right]\\
&=P_o(\beta=\infty)\sum_{k=1}^{\infty}E_o\left[\left.\exp\left\{\langle\theta,X_{H_n}\rangle-\Lambda_a(\theta)H_n\right\},H_n=\tau_k\,\right|\beta=\infty\right]\\
&=P_o(\beta=\infty)\sum_{k=1}^{\infty}E_o\left[\left.\exp\left\{\langle\theta,X_{\tau_k}\rangle-\Lambda_a(\theta)\tau_k\right\},\langle X_{\tau_k},e_1\rangle =n\,\right|\beta=\infty\right]\\
&=P_o(\beta=\infty)\sum_{k=1}^{\infty}\hat{P}_o^\theta\left(\langle Y_k,e_1\rangle =n\right).
\end{align*}
Note that $\hat{P}_o^\theta\left(\langle Y_1,e_1\rangle =1\right)>0$ by (\ref{ellipticity}) and part (a) of Lemma \ref{madabir}. Hence, the desired result follows from the renewal theorem for aperiodic sequences, cf.\ Theorem 10.8 of \cite{Breiman}.
\end{proof}

\begin{lemma}\label{L2lemma}
There exists a $\kappa_{eq}\in(0,\kappa_3)$ such that
$$\sup_{n\geq1}\mathbb{E}\left\{g_n(\theta,\cdot)^2\right\}<\infty$$ for every $\theta\in\mathcal{C}_a(\kappa_{eq})$.
\end{lemma}

\begin{remark}
Lemma \ref{L2lemma} is proved in Section \ref{L2section}.
\end{remark}

\begin{lemma}\label{forconone}
For every $\theta\in\mathcal{C}_a(\kappa_{eq})$,
\begin{equation}\label{notone}
\mathbb{P}\left\{\omega: \lim_{n\to\infty}g_n(\theta,\omega)=0\right\}<1.
\end{equation}
\end{lemma}

\begin{proof}
Take any $\theta\in\mathcal{C}_a(\kappa_{eq})$. Note that $(g_n(\theta,\cdot))_{n\geq1}$ is uniformly integrable by Lemma \ref{L2lemma}.
If $g_n(\theta,\cdot)$ were to converge $\mathbb{P}$-a.s.\ to $0$ as $n\to\infty$, then $\lim_{n\to\infty}\mathbb{E}\left\{g_n(\theta,\cdot)\right\}=0$ would hold. However, this would contradict Lemma \ref{L1lemma}.
\end{proof}

\begin{lemma}\label{forcontwo}
For every $\theta\in\mathbb{R}^d$, $\epsilon>0$ and $\mathbb{P}$-a.e.\ $\omega$,
$$\lim_{n\to\infty}E_o^\omega\left[\exp\left\{\langle\theta,X_{H_n}\rangle-(\Lambda_q(\theta)+\epsilon)H_n\right\}\right]=0.$$
\end{lemma}

\begin{proof}
For every $n\geq1$, $\theta\in\mathbb{R}^d$, $\epsilon>0$ and $\mathbb{P}$-a.e.\ $\omega$,
\begin{align*}
E_o^\omega\left[\exp\left\{\langle\theta,X_{H_n}\rangle-(\Lambda_q(\theta)+\epsilon)H_n\right\}\right]&=\sum_{i=n}^\infty E_o^\omega\left[\exp\left\{\langle\theta,X_{H_n}\rangle-(\Lambda_q(\theta)+\epsilon)H_n\right\},H_n=i\right]\\
&\leq\sum_{i=n}^\infty E_o^\omega\left[\exp\left\{\langle\theta,X_i\rangle-(\Lambda_q(\theta)+\epsilon)i\right\}\right]=\sum_{i=n}^\infty\mathrm{e}^{o(i)-\epsilon i}\leq\sum_{i=n}^\infty\mathrm{e}^{-\epsilon i/2}
\end{align*} when $n$ is sufficiently large. Therefore, 
$$\limsup_{n\to\infty}E_o^\omega\left[\exp\left\{\langle\theta,X_{H_n}\rangle-(\Lambda_q(\theta)+\epsilon)H_n\right\}\right]\leq\limsup_{n\to\infty}\mathrm{e}^{-\epsilon n/2}\left(1-\mathrm{e}^{-\epsilon/2}\right)^{-1}=0.\qedhere$$
\end{proof}

\begin{lemma}\label{loglaresit}
$\Lambda_q(\theta)=\Lambda_a(\theta)$ for every $\theta\in\mathcal{C}_a(\kappa_{eq})$.
\end{lemma}

\begin{proof}
For every $\theta\in\mathbb{R}^d$, it follows from Jensen's inequality and the bounded convergence theorem that
\begin{align*}
\Lambda_q(\theta)&=\mathbb{E}\left\{\lim_{n\to\infty}\frac{1}{n}\log E_o^\omega\left[\exp\{\langle\theta,X_n\rangle\}\right]\right\}=\lim_{n\to\infty}\frac{1}{n}\mathbb{E}\left\{\log E_o^\omega\left[\exp\{\langle\theta,X_n\rangle\}\right]\right\}\\
&\leq\lim_{n\to\infty}\frac{1}{n}\log E_o\left[\exp\{\langle\theta,X_n\rangle\}\right]=\Lambda_a(\theta).
\end{align*}

Let us now establish the reverse inequality. For every $\theta\in\mathcal{C}_a(\kappa_{eq})$ and $\epsilon>0$,
\begin{align*}
&\mathbb{P}\left\{\omega: \lim_{n\to\infty}E_o^\omega\left[\exp\left\{\langle\theta,X_{H_n}\rangle-\Lambda_a(\theta)H_n\right\}\right]=0\right\}<1\quad\mbox{and}\\
&\mathbb{P}\left\{\omega: \lim_{n\to\infty}E_o^\omega\left[\exp\left\{\langle\theta,X_{H_n}\rangle-(\Lambda_q(\theta)+\epsilon)H_n\right\}\right]=0\right\}=1
\end{align*}
by Lemmas \ref{forconone} and \ref{forcontwo}, respectively. Therefore, $\Lambda_q(\theta)+\epsilon>\Lambda_a(\theta)$. Since $\epsilon>0$ is arbitrary, we conclude that $\Lambda_q(\theta)\geq\Lambda_a(\theta)$ for every $\theta\in\mathcal{C}_a(\kappa_{eq})$. 
\end{proof}

\subsection{Equality of the rate functions}\label{guyaispat_iki}

Since $\Lambda_q=\Lambda_a$ on $\mathcal{C}_a(\kappa_{eq})$, it will follow from convex duality that $I_q(\xi)=I_a(\xi)$ for every $\xi\in\mathcal{D}$ that defines a supporting hyperplane of $\Lambda_a$ at some $\theta\in\mathcal{C}_a(\kappa_{eq})$. In order to show that the set of such $\xi$ satisfies the properties stated in Theorem \ref{QequalsA}, we need two preliminary lemmas.

\begin{lemma}\label{geldiiksh}
Assume that the walk is nestling. Define
\begin{equation}\label{nebiliyimiste}
\mathcal{C}_a^b(\kappa_{eq}):= \left\{\theta\in\partial\mathcal{C}_a(\kappa_{eq}):|\theta|<\kappa_{eq}\right\}.
\end{equation}
\begin{itemize}
\item[(a)] If $|\theta|<\kappa_{eq}$, then $\theta\not\in\mathcal{C}_a(\kappa_{eq})$ if and only if $E_o\left[\left.\exp\{\langle\theta,X_{\tau_1}\rangle\}\right|\beta=\infty\right]\leq1$.
\item[(b)] If $|\theta|<\kappa_{eq}$, then $\theta\in\mathcal{C}_a^b(\kappa_{eq})$ if and only if $E_o\left[\left.\exp\{\langle\theta,X_{\tau_1}\rangle\}\right|\beta=\infty\right]=1$. 
\end{itemize}
\end{lemma}

\begin{proof}
This is Lemma 13 of \cite{YilmazAveraged}.
\end{proof}

\begin{lemma}\label{Caveraged}
$\Lambda_a$ is analytic on $\mathcal{C}_a(\kappa_{eq})$. Its gradient $\nabla\Lambda_a$ extends smoothly to $\overline{\mathcal{C}_a(\kappa_{eq})}$, the closure of $\mathcal{C}_a(\kappa_{eq})$. Moreover, the extension of the Hessian $\mathcal{H}_a$ of $\Lambda_a$ is positive definite on $\overline{\mathcal{C}_a(\kappa_{eq})}$.
\end{lemma}

\begin{proof}
This follows immediately from the proof of Lemma 6 of \cite{YilmazAveraged}.
\end{proof}

\begin{proof}[Proof of Theorem \ref{QequalsA}]
\quad
\vspace{0.3cm}

\noindent\textit{(a) The non-nestling case:} Recall that $\Lambda_a$ is analytic on $\mathcal{C}_a(\kappa_{eq})$. Define $\mathcal{A}_{eq}:=\left\{\nabla\Lambda_a(\theta): \theta\in\mathcal{C}_a(\kappa_{eq})\right\}$.
$\nabla\Lambda_a: \mathcal{C}_a(\kappa_{eq})\to\mathcal{A}_{eq}$ is invertible since the Hessian $\mathcal{H}_a$ of $\Lambda_a$ is positive definite on $\mathcal{C}_a(\kappa_{eq})$. The inverse, denoted by $\Gamma_a: \mathcal{A}_{eq}\to\mathcal{C}_a(\kappa_{eq})$, is analytic by the inverse function theorem (cf.\ Theorem 6.1.2 of \cite{Krantz}), and $\mathcal{A}_{eq}$ is open.

For every $\xi\in\mathcal{A}_{eq}$,
\begin{equation}\label{sabretgenco}
I_a(\xi)=\sup_{\theta\in\mathbb{R}^d}\left\{\langle\theta,\xi\rangle-\Lambda_a(\theta)\right\}=\langle\Gamma_a(\xi),\xi\rangle-\Lambda_a(\Gamma_a(\xi)).
\end{equation} Thus, $I_a$ is analytic on $\mathcal{A}_{eq}$. Differentiating (\ref{sabretgenco}) twice with respect to $\xi$ shows that the Hessian of $I_a$ at $\xi$ is equal to $\mathcal{H}_a(\Gamma_a(\xi))^{-1}$, a positive definite matrix. Therefore, $I_a$ is strictly convex on $\mathcal{A}_{eq}$.

It is shown in \cite{SznitmanZerner} that $\xi_o={\left({E_o\left[\left.X_{\tau_1}\right|\beta=\infty\right]}\right)}/{\left({E_o\left[\left.\tau_1\right|\beta=\infty\right]}\right)}$. Since $0\in\mathcal{C}_a(\kappa_{eq})$, it follows that $\xi_o=\nabla\Lambda_a(0)\in\mathcal{A}_{eq}$.

$\Lambda_q=\Lambda_a$ on $\mathcal{C}_a(\kappa_{eq})$ by Lemma \ref{loglaresit}. For every $\xi\in\mathcal{A}_{eq}$,
\begin{equation*}
I_q(\xi)=\sup_{\theta\in\mathbb{R}^d}\left\{\langle\theta,\xi\rangle-\Lambda_q(\theta)\right\}=\langle\Gamma_a(\xi),\xi\rangle-\Lambda_a(\Gamma_a(\xi))=I_a(\xi).
\end{equation*} 
\vspace{0.2cm}

\noindent\textit{(b) The nestling case:} Recall that $\nabla\Lambda_a$ extends smoothly to $\overline{\mathcal{C}_a(\kappa_{eq})}$. Refer to the extension by $\overline{\nabla\Lambda_a}$.
Define $\mathcal{A}_{eq}^+:=\left\{\nabla\Lambda_a(\theta): \theta\in\mathcal{C}_a(\kappa_{eq})\right\}$ and $\mathcal{A}_{eq}^b:=\left\{\overline{\nabla\Lambda_a}(\theta): \theta\in\mathcal{C}_a^b(\kappa_{eq})\right\}$ with $\mathcal{C}_a^b(\kappa_{eq})$ as in (\ref{nebiliyimiste}).
Note that $0\in\mathcal{C}_a^b(\kappa_{eq})\subset\partial\mathcal{C}_a(\kappa_{eq})$ by Lemma \ref{geldiiksh}, and $\xi_o=\overline{\nabla\Lambda_a}(0)\in\mathcal{A}_{eq}^b\subset\partial\mathcal{A}_{eq}^+$.

Similar to the non-nestling case, $I_a$ is strictly convex and analytic on $\mathcal{A}_{eq}^+$ which is an open set, and $I_q(\xi)=I_a(\xi)$ for every $\xi\in\mathcal{A}_{eq}^+$. Moreover, $\mathcal{A}_{eq}^b$ is a $(d-1)$-dimensional smooth surface patch and item (iii) is satisfied. (The latter facts follow from part (d) of Theorem \ref{previousqual} since $\mathcal{A}_{eq}^b\subset\mathcal{A}_a^b$, cf.\ \cite{YilmazAveraged}.)


It remains to show that $I_q(t\xi)=tI_q(\xi)=tI_a(\xi)=I_a(t\xi)$ for every $\xi\in\mathcal{A}_{eq}^b$ and $t\in[0,1]$. The rest of this proof focuses on this statement.

For every $\xi\in\mathcal{A}_{eq}^b$, there exists a $\theta\in\mathcal{C}_a^b(\kappa_{eq})$ such that $\xi=\overline{\nabla\Lambda_a}(\theta)$ and
\begin{equation}\label{yukarikigenc}
\langle\xi,e_1\rangle=\langle\overline{\nabla\Lambda_a}(\theta),e_1\rangle=\frac{E_o\left[\left.\langle X_{\tau_1},e_1\rangle\exp\{\langle\theta,X_{\tau_1}\rangle\}\right|\beta=\infty\right]}{E_o\left[\left.\tau_1\exp\{\langle\theta,X_{\tau_1}\rangle\}\right|\beta=\infty\right]}>0.
\end{equation}
Suppose $\xi=\overline{\nabla\Lambda_a}(\theta')$ for some $\theta'\in\mathcal{C}_a^b(\kappa_{eq})$ such that $\theta\neq\theta'$. Then, for every $t\in(0,1)$, $\xi$ defines a supporting hyperplane of $\Lambda_a$ at $\theta_t:=t\theta+(1-t)\theta'$. Recall Lemma \ref{geldiiksh}. $E_o[\left.\exp\{\langle\theta_t,X_{\tau_1}\rangle\}\right|\beta=\infty]<1$ by Jensen's inequality, and $\theta_t$ is an interior point of $\mathcal{C}_a(\kappa_{eq})^c$. Therefore, $\nabla\Lambda_a(\theta_t)=0$ since $\Lambda_a$ is identically equal to zero on $\{\theta: |\theta|<\kappa_{eq}\}\setminus\mathcal{C}_a(\kappa_{eq})$. However, this contradicts (\ref{yukarikigenc}). We conclude that there exists a \textit{unique} $\theta\in\mathcal{C}_a^b(\kappa_{eq})$ such that $\xi=\overline{\nabla\Lambda_a}(\theta)$. Denote the inverse of $\overline{\nabla\Lambda_a}$ by $\overline{\Gamma_a}$.

For every $\xi\in\mathcal{A}_{eq}^b$ and $t\in[0,1]$, $\exists\,\theta_n\in\mathcal{C}_a(\kappa_{eq})$ such that $\theta_n\to\overline{\Gamma_a}(\xi)$ and $\xi_n:=\nabla\Lambda_a(\theta_n)\to\xi$ as $n\to\infty$. Note that $\Lambda_a(\overline{\Gamma_a}(\xi))=0$ since $\overline{\Gamma_a}(\xi)\in\mathcal{C}_a^b(\kappa_{eq})$. By the continuity of $I_a$ and $\Lambda_a$,
\begin{align*}
I_a(\xi)&=\lim_{n\to\infty}I_a(\xi_n)=\lim_{n\to\infty}\langle\theta_n,\xi_n\rangle - \Lambda_a(\theta_n)=\langle\overline{\Gamma_a}(\xi),\xi\rangle -  \Lambda_a(\overline{\Gamma_a}(\xi)) = \langle\overline{\Gamma_a}(\xi),\xi\rangle\quad\mbox{and}\\
I_a(t\xi) &= \sup_{\theta\in\mathbb{R}^d}\left\{\langle\theta,t\xi\rangle - \Lambda_a(\theta)\right\}\geq\langle\overline{\Gamma_a}(\xi),t\xi\rangle - \Lambda_a(\overline{\Gamma_a}(\xi))=t\langle\overline{\Gamma_a}(\xi),\xi\rangle=tI_a(\xi).
\end{align*}
Conversely, $I_a(t\xi)\leq tI_a(\xi)+(1-t)I_a(0)=tI_a(\xi)$ by Jensen's inequality (and the fact that $I_a(0)=0$, cf.\ Theorem \ref{previousqual}). Hence, $I_a(t\xi)=tI_a(\xi)$.

The continuity of the rate functions implies that $I_q=I_a$ on $\mathcal{A}_{eq}^b$. Recall that $I_q(0)=0$, cf.\ Theorem \ref{previousqual}. Since the averaged rate function is always less than or equal to the quenched rate function, we conclude that
$$I_q(t\xi)\leq tI_q(\xi)+(1-t)I_q(0)=tI_q(\xi)=tI_a(\xi)=I_a(t\xi)\leq I_q(t\xi).\qedhere$$

\end{proof}

\begin{remark}
The argument above, due to its structure, not only proves Theorem \ref{QequalsA}, but also reproduces some of the proofs of the statements in Theorem \ref{previousqual} that are given in \cite{YilmazAveraged}. Moreover, it provides a new and concise proof of item (v) of part (d) of Theorem \ref{previousqual} which is originally obtained in \cite{Jon}.
\end{remark}

\section{The $L^2$ estimate}\label{L2section}

In our proof of Theorem \ref{QequalsA} given in Section \ref{turetmis}, we assumed Lemma \ref{L2lemma}. In this section, we will verify this assumption. The following fact will play a central role in our argument: \textit{if the dimension is at least four, then, with positive averaged probability, the paths of two independent ballistic walks in the same environment do not intersect.}

\subsection{Some preliminaries regarding two walks}

Assume $d\geq4$, (\ref{ellipticity}), (\ref{i.i.d.}) and (\textbf{T},$\hat{u}$) for some $\hat{u}\in\mathcal{S}^{d-1}$.
Also, like in Subsections \ref{guyaispat_bir} and \ref{guyaispat_iki}, assume WLOG that $\langle\xi_o,e_1\rangle>0$.

For every $x$ and $\tilde{x}\in\mathbb{Z}^d$, consider two independent walks $X=X(x):=(X_i)_{i\geq0}$ and $\tilde{X}=\tilde{X}(\tilde{x}):=(\tilde{X}_j)_{j\geq0}$ starting at $x$ and $\tilde{x}$ in the same environment. Denote their joint quenched law and joint averaged law by $P_{x,\tilde{x}}^\omega:=P_x^\omega\!\times\!P_{\tilde{x}}^\omega$ and $P_{x,\tilde{x}}(\cdot):=\mathbb{E}\{P_{x,\tilde{x}}^\omega(\cdot)\}$. As usual, $E_{x,\tilde{x}}^\omega$ and $E_{x,\tilde{x}}$ refer to expectations under $P_{x,\tilde{x}}^\omega$ and $P_{x,\tilde{x}}$, respectively.

Clearly, $P_{x,\tilde{x}}\neq P_x\!\times\!P_{\tilde{x}}$. On the other hand, the two walks don't know that they are in the same environment unless their paths intersect. In particular, for any event $A$ involving $X$ and $\tilde{X}$,
\begin{equation}\label{cokkibar}
P_{x,\tilde{x}}(A\cap\{\nu_1=\infty\})=P_x\!\times\!P_{\tilde{x}}(A\cap\{\nu_1=\infty\})
\end{equation} where 
\begin{equation}\label{hatirlatmak}
\nu_1:=\inf\{m\in\mathbb{Z}:X_i=\tilde{X}_j\mbox{ for some }i\geq0,j\geq0,\mbox{ and }\langle X_i,e_1\rangle=m\}.
\end{equation}

Similar to the random times $(\tau_m)_{m\geq1}=(\tau_m(e_1))_{m\geq1},\ (H_n)_{n\geq0}=(H_n(e_1))_{n\geq0}\mbox{ and }\beta=\beta(e_1)$ defined in (\ref{regenerationtimes_m}) and (\ref{randomtimes}) for $X$, consider $(\tilde{\tau}_m)_{m\geq1}=(\tilde{\tau}_m(e_1))_{m\geq1},\ (\tilde{H}_n)_{n\geq0}=(\tilde{H}_n(e_1))_{n\geq0}\mbox{ and }\tilde{\beta}=\tilde{\beta}(e_1)$ for $\tilde{X}$. In our proof of Lemma \ref{L2lemma}, we will make use of the \textit{joint regeneration levels} of $X$ and $\tilde{X}$, which are elements of 
$$\mathcal{L}:=\{n\geq0: \langle X_i,e_1\rangle\geq n\mbox{ and }\langle\tilde{X}_j,e_1\rangle\geq n\mbox{ for every }i\geq H_n\mbox{ and }j\geq\tilde{H}_n\}.$$
This random set has been previously introduced and studied by Rassoul-Agha and Sepp\"al\"ainen \cite{TimoFirasCLT}. Note that if the starting points $x$ and $\tilde{x}$ are both in $\mathbb{V}_d:=\left\{z\in\mathbb{Z}^d: \langle z,e_1\rangle=0\right\}$, then $$0\in\mathcal{L}\iff\beta=\tilde{\beta}=\infty\iff l_1:=\inf\mathcal{L}=0.$$

Let $\mathbb{V}_d':=\mathbb{V}_d\setminus\{0\}$. As mentioned in the opening paragraph of this section, the following lemma is central to our proof of Lemma \ref{L2lemma}.

\begin{lemma}[Berger and Zeitouni \cite{NoamOfer}, Proposition 3.4]\label{nonintersectlemma}
$$\inf_{z\in\mathbb{V}_d'}P_{o,z}(l_1=0)\geq\inf_{z\in\mathbb{V}_d'}P_{o,z}(\nu_1=\infty, l_1=0)>0.$$
\end{lemma}

\noindent The proof of Lemma \ref{nonintersectlemma} is based on certain Green's function estimates which fail to hold unless $d\geq4$.

\subsection{A renewal argument}\label{genclik}

For every $n\geq1$, $\theta\in\mathcal{C}_a(\kappa_3)$, $z\in\mathbb{V}_d$ and $\omega\in\Omega$,
\begin{align*}
g_n(\theta,\omega)&=E_o^\omega\left[\exp\left\{\langle\theta,X_{H_n}\rangle-\Lambda_a(\theta)H_n\right\},H_n=\tau_k\ \mbox{for some }k\geq1, \beta=\infty\right]\quad\mbox{and}\\
g_n(\theta,T_z\omega)&=E_o^{T_z\omega}\left[\exp\left\{\langle\theta,X_{H_n}\rangle-\Lambda_a(\theta)H_n\right\},H_n=\tau_k\ \mbox{for some }k\geq1, \beta=\infty\right]\\
&=\mathrm{e}^{-\langle\theta,z\rangle}E_z^\omega\left[\exp\left\{\langle\theta,X_{H_n}\rangle-\Lambda_a(\theta)H_n\right\},H_n=\tau_k\ \mbox{for some }k\geq1, \beta=\infty\right].
\end{align*}
Thus,
\begin{align*}
G_{n,z}(\theta)&:=\mathbb{E}\left\{g_n(\theta,\cdot)g_n(\theta,T_z\cdot)\right\}=\mathrm{e}^{-\langle\theta,z\rangle}E_{o,z}\left[f(\theta,n,X,\tilde{X}), n\in\mathcal{L}, l_1=0\right]\quad\mbox{where}\\
f(\theta,n,X,\tilde{X})&:=\exp\left\{\langle\theta,X_{H_n}\rangle-\Lambda_a(\theta)H_n\right\}\exp\{\langle\theta,\tilde{X}_{\tilde{H}_n}\rangle-\Lambda_a(\theta)\tilde{H}_n\}.
\end{align*}
Our aim is to show that $\left(G_{n,o}(\theta)\right)_{n\geq1}$ is  bounded. We start the argument by considering a related family of functions $\left(F_{n,z}(\theta)\right)_{n\geq1,z\in\mathbb{V}_d'}$ where
$$F_{n,z}(\theta):=\mathrm{e}^{-\langle\theta,z\rangle}E_{o,z}\left[\left.f(\theta,n,X,\tilde{X}), n\in\mathcal{L}, X_{H_n}\neq\tilde{X}_{\tilde{H}_n}\right| l_1=0\right].$$
Recall (\ref{hatirlatmak}). It follows from the definitions and the regeneration structure that
\begin{align*}
F_{n,z}(\theta)&=\sum_{k=1}^{n}\sum_{z'\in\mathbb{V}_d'}\mathrm{e}^{-\langle\theta,z\rangle}E_{o,z}\left[\left.f(\theta,n,X,\tilde{X}),k=\inf\{l\in\mathcal{L}:l>\nu_1, X_{H_l}\neq\tilde{X}_{\tilde{H}_l}\},\tilde{X}_{\tilde{H}_k}-X_{H_k}=z',\right.\right.\\
&\hspace{10.9cm}\left.\left.n\in\mathcal{L}, X_{H_n}\neq\tilde{X}_{\tilde{H}_n}\right| l_1=0\right]\\
&\quad+\mathrm{e}^{-\langle\theta,z\rangle}E_{o,z}\left[\left.f(\theta,n,X,\tilde{X}), n\leq\nu_1, n\in\mathcal{L}, X_{H_n}\neq\tilde{X}_{\tilde{H}_n}\right| l_1=0\right]\\
&=\sum_{k=1}^{n}\sum_{z'\in\mathbb{V}_d'}\mathrm{e}^{-\langle\theta,z\rangle}E_{o,z}\left[\left.f(\theta,k,X,\tilde{X}),k=\inf\{l\in\mathcal{L}:l>\nu_1, X_{H_l}\neq\tilde{X}_{\tilde{H}_l}\},\tilde{X}_{\tilde{H}_k}-X_{H_k}=z'\right| l_1=0\right]\\
&\qquad\qquad\quad\times\mathrm{e}^{-\langle\theta,z'\rangle}E_{o,z'}\left[\left.f(\theta,n-k,X,\tilde{X}),n-k\in\mathcal{L}, X_{H_{n-k}}\neq\tilde{X}_{\tilde{H}_{n-k}}\right| l_1=0\right]\\
&\quad+\mathrm{e}^{-\langle\theta,z\rangle}E_{o,z}\left[\left.f(\theta,n,X,\tilde{X}), n\leq\nu_1, n\in\mathcal{L}, X_{H_n}\neq\tilde{X}_{\tilde{H}_n}\right| l_1=0\right].
\end{align*}
Therefore,
\begin{align*}
F_{n,z}(\theta)&\leq\sum_{k=1}^{n}\mathrm{e}^{-\langle\theta,z\rangle}E_{o,z}\left[\left.f(\theta,k,X,\tilde{X}),k=\inf\{l\in\mathcal{L}:l>\nu_1, X_{H_l}\neq\tilde{X}_{\tilde{H}_l}\}\right| l_1=0\right]\sup_{z'\in\mathbb{V}_d'}F_{n-k,z'}(\theta)\\
&\quad+\mathrm{e}^{-\langle\theta,z\rangle}E_{o,z}\left[\left.f(\theta,n,X,\tilde{X}), n\leq\nu_1, n\in\mathcal{L}, X_{H_n}\neq\tilde{X}_{\tilde{H}_n}\right| l_1=0\right].
\end{align*}
In other words, 
$$F_{n,z}(\theta)\leq\sum_{k=1}^{n}B_{k,z}(\theta)\sup_{z'\in\mathbb{V}_d'}F_{n-k,z'}(\theta)+C_{n,z}(\theta)$$
where
\begin{align*}
B_{k,z}(\theta)&:=\mathrm{e}^{-\langle\theta,z\rangle}E_{o,z}\left[\left.f(\theta,k,X,\tilde{X}),k=\inf\{l\in\mathcal{L}:l>\nu_1, X_{H_l}\neq\tilde{X}_{\tilde{H}_l}\}\right| l_1=0\right]\quad\mbox{and}\\
C_{n,z}(\theta)&:=\mathrm{e}^{-\langle\theta,z\rangle}E_{o,z}\left[\left.f(\theta,n,X,\tilde{X}), n\leq\nu_1, n\in\mathcal{L}, X_{H_n}\neq\tilde{X}_{\tilde{H}_n}\right| l_1=0\right].
\end{align*}

\begin{lemma}\label{mevzubu}
There exists a $\kappa_{eq}\in(0,\kappa_3)$ such that
$$\mbox{(a)}\ \ C(\theta):=\sup_{n\geq1}\sup_{z\in\mathbb{V}_d'}C_{n,z}(\theta)<\infty\qquad\mbox{and}\qquad\mbox{(b)}\ \ B(\theta):=\sup_{z\in\mathbb{V}_d'}\sum_{k=1}^{\infty}B_{k,z}(\theta)<1$$ for every $\theta\in\mathcal{C}_a(\kappa_{eq})$.
\end{lemma}

\begin{remark}
Lemma \ref{mevzubu} is proved in Subsection \ref{noninterSection}.
\end{remark}

\begin{lemma}\label{kombinezon}
For every $\theta\in\mathcal{C}_a(\kappa_{eq})$,
$$\sup_{n\geq1}\sup_{z\in\mathbb{V}_d'}F_{n,z}(\theta)<\infty.$$ 
\end{lemma}

\begin{proof}
For every $n\geq1$, $N\geq n$ and $z\in\mathbb{V}_d'$,
\begin{align*}
F_{n,z}(\theta)&\leq\sum_{k=1}^{n}B_{k,z}(\theta)\sup_{z'\in\mathbb{V}_d'}F_{n-k,z'}(\theta)+C_{n,z}(\theta)\leq\left(\sum_{k=1}^{n}B_{k,z}(\theta)\right)\sup_{m\leq N}\sup_{z'\in\mathbb{V}_d'}F_{m,z'}(\theta)+C_{n,z}(\theta)\\
&\leq B(\theta)\sup_{m\leq N}\sup_{z'\in\mathbb{V}_d'}F_{m,z'}(\theta)+C(\theta).
\end{align*}
Therefore, 
$$\sup_{n\leq N}\sup_{z\in\mathbb{V}_d'}F_{n,z}(\theta)\leq B(\theta)\sup_{n\leq N}\sup_{z\in\mathbb{V}_d'}F_{n,z}(\theta)+C(\theta).$$
Finally, by Lemma \ref{mevzubu},
$$\sup_{n\geq1}\sup_{z\in\mathbb{V}_d'}F_{n,z}(\theta)\leq C(\theta)\left(1-B(\theta)\right)^{-1}<\infty.\qedhere$$
\end{proof}

\begin{proof}[Proof of Lemma \ref{L2lemma}]
For every $n\geq1$, $\theta\in\mathcal{C}_a(\kappa_{eq})$ and $z\in\mathbb{V}_d'$,
\begin{align*}
&F_{n+1,z}(\theta)=\mathrm{e}^{-\langle\theta,z\rangle}E_{o,z}\left[\left.f(\theta,n+1,X,\tilde{X}), n+1\in\mathcal{L}, X_{H_{n+1}}\neq\tilde{X}_{\tilde{H}_{n+1}}\right| l_1=0\right]\\
&\quad\geq\mathrm{e}^{-\langle\theta,z\rangle}E_{o,z}\left[\left.f(\theta,n+1,X,\tilde{X}), n\in\mathcal{L}, X_{H_n}=\tilde{X}_{\tilde{H}_n}, n+1\in\mathcal{L}, X_{H_{n+1}}\neq\tilde{X}_{\tilde{H}_{n+1}}\right| l_1=0\right]\\
&\quad=\mathrm{e}^{-\langle\theta,z\rangle}E_{o,z}\left[\left.f(\theta,n,X,\tilde{X}), n\in\mathcal{L}, X_{H_n}=\tilde{X}_{\tilde{H}_n} \right| l_1=0\right]E_{o,o}\left[\left.f(\theta,1,X,\tilde{X}),1\in\mathcal{L}, X_{H_1}\neq\tilde{X}_{\tilde{H}_1}\right| l_1=0\right].
\end{align*}
Therefore,
\begin{align}
\frac{G_{n,z}(\theta)}{P_{o,z}(l_1=0)}-F_{n,z}(\theta)&=\mathrm{e}^{-\langle\theta,z\rangle}E_{o,z}\left[\left.f(\theta,n,X,\tilde{X}), n\in\mathcal{L}, X_{H_n}=\tilde{X}_{\tilde{H}_n} \right| l_1=0\right]\nonumber\\
&\leq E_{o,o}\left[\left.f(\theta,1,X,\tilde{X}),1\in\mathcal{L}, X_{H_1}\neq\tilde{X}_{\tilde{H}_1}\right| l_1=0\right]^{-1}F_{n+1,z}(\theta).\label{klirlibdd}
\end{align}
By the uniform ellipticity assumption (\ref{ellipticity}), Lemma \ref{nonintersectlemma}, and part (a) of Lemma \ref{madabir}, the first term in (\ref{klirlibdd}) is bounded from above. This, in combination with Lemma \ref{kombinezon}, implies that
$$\sup_{n\geq1}\sup_{z\in\mathbb{V}_d'}G_{n,z}(\theta)<\infty.$$

For every $\hat{z}\in U\cap\mathbb{V}_d$, 
\begin{align*}
g_n(\theta,T_{\hat{z}}\omega)&=\mathrm{e}^{-\langle\theta,\hat{z}\rangle}E_{\hat{z}}^\omega\left[\exp\left\{\langle\theta,X_{H_n}\rangle-\Lambda_a(\theta)H_n\right\},H_n=\tau_k\ \mbox{for some }k\geq1, \beta=\infty\right]\\
&\geq\mathrm{e}^{-\langle\theta,\hat{z}\rangle}E_{\hat{z}}^\omega\left[\exp\left\{\langle\theta,X_{H_n}\rangle-\Lambda_a(\theta)H_n\right\},H_n=\tau_k\ \mbox{for some }k\geq1, X_1=0, \beta=\infty\right]\\
&\geq\delta\mathrm{e}^{-\langle\theta,\hat{z}\rangle-\Lambda_a(\theta)}E_o^\omega\left[\exp\left\{\langle\theta,X_{H_n}\rangle-\Lambda_a(\theta)H_n\right\},H_n=\tau_k\ \mbox{for some }k\geq1, \beta=\infty\right]\\
&=\delta\mathrm{e}^{-\langle\theta,\hat{z}\rangle-\Lambda_a(\theta)}g_n(\theta,\omega)\geq\delta\mathrm{e}^{-2\kappa_{eq}}g_n(\theta,\omega).
\end{align*}
Hence,
$$\sup_{n\geq1}\mathbb{E}\left\{g_n(\theta,\cdot)^2\right\}\leq\delta^{-1}\mathrm{e}^{2\kappa_{eq}}\sup_{n\geq1}\sup_{z\in\mathbb{V}_d'}\mathbb{E}\left\{g_n(\theta,\cdot)g_n(\theta,T_z\cdot)\right\}=\delta^{-1}\mathrm{e}^{2\kappa_{eq}}\sup_{n\geq1}\sup_{z\in\mathbb{V}_d'}G_{n,z}(\theta)<\infty.\qedhere$$
\end{proof}

\subsection{Proof of Lemma \ref{mevzubu}}\label{noninterSection}

Let us start by proving the easy part.
\begin{proof}[Proof of part (a) of Lemma \ref{mevzubu}]
For every $n\geq1$, $\theta\in\mathcal{C}_a(\kappa_3)$ and $z\in\mathbb{V}_d'$,
\begin{align}
C_{n,z}(\theta)&\leq\frac{\mathrm{e}^{-\langle\theta,z\rangle}}{P_{o,z}(l_1=0)}E_{o,z}\left[f(\theta,n,X,\tilde{X}), n\leq\nu_1, n\in\mathcal{L}, \beta=\infty, \tilde{\beta}=\infty\right]\nonumber\\
&\leq\frac{\mathrm{e}^{-\langle\theta,z\rangle}}{P_{o,z}(l_1=0)}E_{o,z}\left[f(\theta,n,X,\tilde{X}), \{X_i:0\leq i<H_n\}\cap\{\tilde{X}_j:0\leq j<\tilde{H}_n\}=\emptyset, \beta\geq H_n, \tilde{\beta}\geq\tilde{H}_n\right]\nonumber\\
&=\frac{\mathrm{e}^{-\langle\theta,z\rangle}}{P_{o,z}(l_1=0)}E_o\!\times\!E_z\left[f(\theta,n,X,\tilde{X}), \{X_i:0\leq i<H_n\}\cap\{\tilde{X}_j:0\leq j<\tilde{H}_n\}=\emptyset, \beta\geq H_n, \tilde{\beta}\geq\tilde{H}_n\right]\label{coupling}\\
&\leq\frac{\mathrm{e}^{-\langle\theta,z\rangle}}{P_{o,z}(l_1=0)}E_o\!\times\!E_z\left[f(\theta,n,X,\tilde{X}), \beta\geq H_n, \tilde{\beta}\geq\tilde{H}_n\right]\nonumber\\
&=\frac{1}{P_{o,z}(l_1=0)}E_o\left[\exp\left\{\langle\theta,X_{H_n}\rangle-\Lambda_a(\theta)H_n\right\}, \beta\geq H_n\right]^2.\label{birisimver}
\end{align}
Here, (\ref{coupling}) is similar to (\ref{cokkibar}). Both facts follow from a standard coupling argument (cf.\ \cite{NoamOfer}, Proposition 3.7.) Note that
\begin{align*}
&E_o\left[\exp\left\{\langle\theta,X_{H_n}\rangle-\Lambda_a(\theta)H_n\right\}, \beta\geq H_n\right]=\frac{P_o(\beta=\infty)}{P_o(\beta=\infty)}E_o\left[\exp\left\{\langle\theta,X_{H_n}\rangle-\Lambda_a(\theta)H_n\right\}, \beta\geq H_n\right]\\
&\qquad=\frac{1}{P_o(\beta=\infty)}E_o\left[\exp\left\{\langle\theta,X_{H_n}\rangle-\Lambda_a(\theta)H_n\right\}, \beta\geq H_n\right]P_o(\langle X_i,e_1\rangle\geq n\mbox{ for every }i\geq H_n)\\
&\qquad=\frac{1}{P_o(\beta=\infty)}E_o\left[\exp\left\{\langle\theta,X_{H_n}\rangle-\Lambda_a(\theta)H_n\right\}, H_n=\tau_k\mbox{ for some }k\geq1, \beta=\infty\right]\\
&\qquad={\mathbb{E}\{g_n(\theta,\cdot)\}}/{P_o(\beta=\infty)}.
\end{align*}
Therefore, (\ref{birisimver}), Lemma \ref{L1lemma} and Lemma \ref{nonintersectlemma} imply that
$$\sup_{n\geq1}\sup_{z\in\mathbb{V}_d'}C_{n,z}(\theta)\leq P_o(\beta=\infty)^{-2}\sup_{z\in\mathbb{V}_d'}P_{o,z}(l_1=0)^{-1}\left(\sup_{n\geq1}\mathbb{E}\{g_n(\theta,\cdot)\}\right)^2<\infty.\qedhere$$
\end{proof}

The proof of part (b) of Lemma \ref{mevzubu} is more technical. At $\theta=0$, 
\begin{align}
B(0)&=\sup_{z\in\mathbb{V}_d'}\sum_{k=1}^{\infty}B_{k,z}(0)=\sup_{z\in\mathbb{V}_d'}\sum_{k=1}^{\infty}P_{o,z}\left(\left.k=\inf\{l\in\mathcal{L}:l>\nu_1, X_{H_l}\neq\tilde{X}_{\tilde{H}_l}\}\right|l_1=0\right)\nonumber\\
&=\sup_{z\in\mathbb{V}_d'}P_{o,z}\left(\left.\nu_1<\infty\right|l_1=0\right)=1-\inf_{z\in\mathbb{V}_d'}P_{o,z}\left(\left.\nu_1=\infty\right|l_1=0\right)<1\label{birdenkucuk}
\end{align} by Lemma \ref{nonintersectlemma}. For every $\theta\in\mathcal{C}_a(\kappa_3)$ and $z\in\mathbb{V}_d'$,
\begin{align}
\sum_{k=1}^{\infty}B_{k,z}(\theta)&=\sum_{k=1}^{\infty}B_{k,z}(0)+\sum_{k=1}^{\infty}\left(B_{k,z}(\theta)-B_{k,z}(0)\right)\quad\mbox{and}\nonumber\\
B(\theta)=\sup_{z\in\mathbb{V}_d'}\sum_{k=1}^{\infty}B_{k,z}(\theta)&\leq\sup_{z\in\mathbb{V}_d'}\sum_{k=1}^{\infty}B_{k,z}(0)+\sum_{k=1}^{\infty}\sup_{z\in\mathbb{V}_d'}\left(B_{k,z}(\theta)-B_{k,z}(0)\right)\nonumber\\
&=B(0)+\sum_{k=1}^{\infty}\sup_{z\in\mathbb{V}_d'}\left(B_{k,z}(\theta)-B_{k,z}(0)\right).\label{kontrolet}
\end{align}
The next three lemmas control the sum in (\ref{kontrolet}).

\begin{lemma}\label{kontrolediyorbir}
For every $k\geq1$ and $\epsilon>0$, there exists a $\kappa_4=\kappa_4(k,\epsilon)\in(0,\kappa_3)$ such that
\begin{equation}\label{ayseguldu}
\sup_{z\in\mathbb{V}_d'}\sup_{\theta\in\mathcal{C}_a(\kappa_4)}\left(B_{k,z}(\theta)-B_{k,z}(0)\right)<\epsilon.
\end{equation}
\end{lemma}

\begin{proof}
For every $k\geq1$, $\theta\in\mathcal{C}_a(\kappa_3)$ and $z\in\mathbb{V}_d'$,
\begin{align}
&B_{k,z}(\theta)-B_{k,z}(0)=E_{o,z}\left[\left.\mathrm{e}^{-\langle\theta,z\rangle}f(\theta,k,X,\tilde{X})-1,k=\inf\{l\in\mathcal{L}:l>\nu_1, X_{H_l}\neq\tilde{X}_{\tilde{H}_l}\}\right| l_1=0\right]\nonumber\\
&\qquad\leq P_{o,z}(l_1=0)^{-1}E_{o,z}\left[\left(\mathrm{e}^{-\langle\theta,z\rangle}f(\theta,k,X,\tilde{X})-1\right)^2\right]^{1/2}P_{o,z}\left(k=\inf\{l\in\mathcal{L}:l>\nu_1, X_{H_l}\neq\tilde{X}_{\tilde{H}_l}\}\right)^{1/2}\nonumber\\
&\qquad\leq P_{o,z}(l_1=0)^{-1}\mathbb{E}\left\{E_o^\omega\!\times\!E_o^{T_z\omega}\left[\left(f(\theta,k,X,\tilde{X})-1\right)^2\right]\right\}^{1/2}\nonumber\\
&\qquad\qquad\times\left(P_{o,z}\left(\sup_{1\leq i\leq H_k}\left|X_i\right|\geq\frac{|z|}{2}\right)+P_{o,z}\left(\sup_{1\leq j\leq\tilde{H}_k}\left|\tilde{X}_j-z\right|\geq\frac{|z|}{2}\right)\right)^{1/2}\nonumber\\
&\qquad\leq P_{o,z}(l_1=0)^{-1}\left(\mathbb{E}\left\{E_o^\omega\!\times\!E_o^{T_z\omega}\left[f(\theta,k,X,\tilde{X})^2\right]\right\}+1\right)^{1/2}\sqrt2P_o\left(\sup_{1\leq i\leq H_k}\left|X_i\right|\geq\frac{|z|}{2}\right)^{1/2}\nonumber\\
&\qquad\leq P_{o,z}(l_1=0)^{-1}\left(E_o\left[\exp\left\{4\langle\theta,X_{H_k}\rangle-4\Lambda_a(\theta)H_k\right\}\right]+1\right)^{1/2}\sqrt2P_o\left(\sup_{1\leq i\leq\tau_k}\left|X_i\right|\geq\frac{|z|}{2}\right)^{1/2}.\label{esregnet}
\end{align}
For every $\epsilon>0$, it follows from (\ref{esregnet}), Corollary \ref{mazaltov} and Lemma \ref{nonintersectlemma} that there exists an $N\geq1$ such that
\begin{equation}\label{neyapiyorsimdi}
\sup_{{z\in\mathbb{V}_d'}\atop{|z|>N}}\sup_{\theta\in\mathcal{C}_a(\kappa_3/4)}\left(B_{k,z}(\theta)-B_{k,z}(0)\right)<\epsilon.
\end{equation}

Note that $\theta\mapsto f(\theta,k,X,\tilde{X})$ is continuous. Hence, for every $k\geq1$ and $z\in\mathbb{V}_d'$, the map $\theta\mapsto B_{k,z}(\theta)$ is continuous at $0$ by Schwarz's inequality, Corollary \ref{mazaltov} and the dominated convergence theorem. Consequently, there exists a $\kappa_4\in(0,\kappa_3/4)$ such that
\begin{equation}\label{neyapiyorsimdilik}
\sup_{{z\in\mathbb{V}_d'}\atop{|z|\leq N}}\sup_{\theta\in\mathcal{C}_a(\kappa_4)}\left|B_{k,z}(\theta)-B_{k,z}(0)\right|<\epsilon.
\end{equation}
Clearly, (\ref{neyapiyorsimdi}) and (\ref{neyapiyorsimdilik}) imply (\ref{ayseguldu}).
\end{proof}

\begin{lemma}\label{kontrolediyoriki}
There exists a $\kappa_5\in(0,\kappa_3)$ such that
$$\sum_{k=1}^\infty\sup_{z\in\mathbb{V}_d'}\sup_{\theta\in\mathcal{C}_a(\kappa_5)}B_{k,z}(\theta)<\infty.$$
\end{lemma}

\begin{proof}
For every $k\geq1$, $\kappa\in(0,\kappa_3)$, $\theta\in\mathcal{C}_a(\kappa)$ and $z\in\mathbb{V}_d'$,
\begin{align}
B_{k,z}(\theta)&=\mathrm{e}^{-\langle\theta,z\rangle}E_{o,z}\left[\left.f(\theta,k,X,\tilde{X}),k=\inf\{l\in\mathcal{L}:l>\nu_1, X_{H_l}\neq\tilde{X}_{\tilde{H}_l}\}\right|l_1=0\right]\nonumber\\
&=\sum_{j=0}^{k-1}\sum_{z'\in\mathbb{V}_d}\mathrm{e}^{-\langle\theta,z\rangle}E_{o,z}\left[\left.f(\theta,k,X,\tilde{X}),j=\sup\{l\in\mathcal{L}:l\leq\nu_1\},\tilde{X}_{\tilde{H}_j}-X_{H_j}=z',\right.\right.\nonumber\\
&\hspace{6.7cm}\left.\left.k=\inf\{l\in\mathcal{L}:l>\nu_1, X_{H_l}\neq\tilde{X}_{\tilde{H}_l}\}\right|l_1=0\right]\nonumber\\
&\leq\sum_{j=0}^{k-1}\sum_{z'\in\mathbb{V}_d}\mathrm{e}^{-\langle\theta,z\rangle}E_{o,z}\left[\left.f(\theta,j,X,\tilde{X}),j\in\mathcal{L},j\leq\nu_1,\tilde{X}_{\tilde{H}_j}-X_{H_j}=z'\right| l_1=0\right]\nonumber\\
&\hspace{1.0cm}\times\mathrm{e}^{-\langle\theta,z'\rangle}E_{o,z'}\left[\left.f(\theta,k-j,X,\tilde{X}),k-j=\inf\{l\in\mathcal{L}:l>0, X_{H_l}\neq\tilde{X}_{\tilde{H}_l}\},k-j>\nu_1\right|l_1=0\right]\nonumber\\
&=\sum_{j=0}^{k-1}\sum_{z'\in\mathbb{V}_d}\mathrm{e}^{-\langle\theta,z\rangle}E_{o,z}\left[\left.f(\theta,j,X,\tilde{X}),j\in\mathcal{L},j\leq\nu_1,\tilde{X}_{\tilde{H}_j}-X_{H_j}=z'\right| l_1=0\right]h_{k-j,z'}(\theta)\nonumber\\
&\leq\sum_{j=0}^{k-1}\sup_{z'\in\mathbb{V}_d}\mathrm{e}^{-\langle\theta,z\rangle}E_{o,z}\left[\left.f(\theta,j,X,\tilde{X}),j\in\mathcal{L},j\leq\nu_1,\tilde{X}_{\tilde{H}_j}-X_{H_j}=z'\right| l_1=0\right]\sum_{z'\in\mathbb{V}_d}h_{k-j,z'}(\theta)\label{ersubigger}
\end{align}
where, for every $i\geq1$,
\begin{equation}\label{reginbogin}
h_{i,z'}(\theta):=\mathrm{e}^{-\langle\theta,z'\rangle}E_{o,z'}\left[\left.f(\theta,i,X,\tilde{X}),i=\inf\{l\in\mathcal{L}:l>0, X_{H_l}\neq\tilde{X}_{\tilde{H}_l}\},i>\nu_1\right| l_1=0\right].
\end{equation}
For every $z'\in\mathbb{V}_d$,
\begin{align}
&\mathrm{e}^{-\langle\theta,z\rangle}E_{o,z}\left[f(\theta,j,X,\tilde{X}),j\in\mathcal{L},j\leq\nu_1,\tilde{X}_{\tilde{H}_j}-X_{H_j}=z', l_1=0\right]\nonumber\\
&\qquad\qquad\leq\mathrm{e}^{-\langle\theta,z\rangle}E_{o,z}\left[f(\theta,j,X,\tilde{X}),\{X_n:0\leq n<H_j\}\cap\{\tilde{X}_m:0\leq m<\tilde{H}_j\}=\emptyset,\right.\nonumber\\
&\hspace{7.7cm}\left.\tilde{X}_{\tilde{H}_j}-X_{H_j}=z',\beta\geq H_j, \tilde{\beta}\geq\tilde{H}_j\right]\nonumber\\
&\qquad\qquad=\mathrm{e}^{-\langle\theta,z\rangle}E_o\!\times\!E_z\left[f(\theta,j,X,\tilde{X}),\{X_n:0\leq n<H_j\}\cap\{\tilde{X}_m:0\leq m<\tilde{H}_j\}=\emptyset,\right.\nonumber\\
&\hspace{8.2cm}\left.\tilde{X}_{\tilde{H}_j}-X_{H_j}=z',\beta\geq H_j, \tilde{\beta}\geq\tilde{H}_j\right]\nonumber\\
&\qquad\qquad\leq\mathrm{e}^{-\langle\theta,z\rangle}E_o\!\times\! E_z\left[f(\theta,j,X,\tilde{X}),\tilde{X}_{\tilde{H}_j}-X_{H_j}=z', \beta\geq H_j, \tilde{\beta}\geq\tilde{H}_j\right]\nonumber\\
&\qquad\qquad=E_o\!\times\! E_o\left[f(\theta,j,X,\tilde{X}),\tilde{X}_{\tilde{H}_j}-X_{H_j}=z'-z, \beta\geq H_j, \tilde{\beta}\geq\tilde{H}_j\right]\nonumber\\
&\qquad\qquad=\frac{P_o\!\times\! P_o(j\in\mathcal{L})}{P_o\!\times\! P_o(l_1=0)}E_o\!\times\! E_o\left[f(\theta,j,X,\tilde{X}),\tilde{X}_{\tilde{H}_j}-X_{H_j}=z'-z, \beta\geq H_j, \tilde{\beta}\geq\tilde{H}_j\right]\nonumber\\
&\qquad\qquad=E_o\!\times\! E_o\left[\left.f(\theta,j,X,\tilde{X}),j\in\mathcal{L},\tilde{X}_{\tilde{H}_j}-X_{H_j}=z'-z\right| l_1=0\right]\nonumber\\
&\qquad\qquad=\hat{P}_o^\theta\!\times\!\hat{P}_o^\theta\left(\exists\,n,m\mbox{ such that }\langle Y_n,e_1\rangle=j\mbox{ and }\tilde{Y}_m-Y_n=z'-z\right)\label{adivar}
\end{align}
where $\left(Y_n\right)_{n\geq0}$ and $\left(\tilde{Y}_m\right)_{m\geq0}$ denote two independent random walks on $\mathbb{Z}^d$, both with transition kernel $q^\theta(y)_{y\in\mathbb{Z}^d}$ given in (\ref{immydit}).

Let $\mu=\mu(\theta):=\hat{E}_o^\theta[\langle Y_1,e_1\rangle]$. For every $j\geq1$, (\ref{adivar}) is equal to
\begin{align}
&\sum_{\langle x,e_1\rangle=j}\hat{P}_o^\theta\left(\exists\,n\mbox{ such that } Y_n=x\right)\hat{P}_o^\theta\left(\exists\,m\mbox{ such that } \tilde{Y}_m=x+z'-z\right)\nonumber\\
&\qquad\leq\sup_{\langle x,e_1\rangle=j}\hat{P}_o^\theta\left(\exists\,n\mbox{ such that } Y_n=x\right)\sum_{\langle x,e_1\rangle=j}\hat{P}_o^\theta\left(\exists\,m\mbox{ such that } \tilde{Y}_m=x+z'-z\right)\nonumber\\
&\qquad=\sup_{\langle x,e_1\rangle=j}\sum_{n\geq1}\hat{P}_o^\theta\left(Y_n=x\right)\hat{P}_o^\theta\left(\exists\,m\mbox{ such that } \langle\tilde{Y}_m,e_1\rangle=j\right)\nonumber\\
&\qquad\leq\sup_{\langle x,e_1\rangle=j}\sum_{n\geq1}\hat{P}_o^\theta\left(Y_n=x\right)\nonumber\\
&\qquad=\sup_{\langle x,e_1\rangle=j}\left(\sum_{|n-j/\mu|\leq\sqrt{j/\mu}}\!\!\!\!\!\hat{P}_o^\theta\left(Y_n=x\right)\ \ +\!\!\!\!\!\!\!\sum_{|n-j/\mu|>\sqrt{j/\mu}}\!\!\!\!\!\hat{P}_o^\theta\left(Y_n=x\right)\right)\nonumber\\
&\qquad\leq S(\theta)j^{-(d-1)/2}.\label{adnan}
\end{align}
Here, (\ref{adnan}) follows from (\ref{onedaylate}) and the local CLT. $S(\theta)$ depends on the mean and covariance of $\left(q^\theta(y)\right)_{y\in\mathbb{Z}^d}$. In particular, $\sup_{\theta\in\mathcal{C}_a(\kappa_3)}S(\theta)<\infty$.

Putting (\ref{ersubigger}), (\ref{adivar}) and (\ref{adnan}) together, we see that
\begin{align*}
\sup_{z\in\mathbb{V}_d'}\sup_{\theta\in\mathcal{C}_a(\kappa)}B_{k,z}(\theta)&\leq\frac{\sup_{\theta\in\mathcal{C}_a(\kappa)}S(\theta)}{\inf_{z\in\mathbb{V}_d'}P_{o,z}(l_1=0)}\sum_{j=0}^{k-1}\max(1,j)^{-(d-1)/2}\sum_{z'\in\mathbb{V}_d}\sup_{\theta\in\mathcal{C}_a(\kappa)}h_{k-j,z'}(\theta)\quad\mbox{and}\\
\sum_{k=1}^\infty\sup_{z\in\mathbb{V}_d'}\sup_{\theta\in\mathcal{C}_a(\kappa)}B_{k,z}(\theta)&\leq\frac{\sup_{\theta\in\mathcal{C}_a(\kappa)}S(\theta)}{\inf_{z\in\mathbb{V}_d'}P_{o,z}(l_1=0)}\left(1+\sum_{j=1}^{\infty}j^{-(d-1)/2}\right)\sum_{i=1}^\infty\sum_{z'\in\mathbb{V}_d}\sup_{\theta\in\mathcal{C}_a(\kappa)}h_{i,z'}(\theta).
\end{align*}
The desired result follows from Lemma \ref{nonintersectlemma} and Lemma \ref{jointreglemma} (stated below.)
\end{proof}

\begin{lemma}\label{jointreglemma}
Recall (\ref{reginbogin}). There exists a $\kappa_5\in(0,\kappa_3)$ such that
$$\sum_{i=1}^\infty\sum_{z\in\mathbb{V}_d}\sup_{\theta\in\mathcal{C}_a(\kappa_5)}h_{i,z}(\theta)<\infty.$$
\end{lemma}

\begin{remark}
Lemma \ref{jointreglemma} is proved in Subsection \ref{jointregsection}.
\end{remark}

Finally, we are ready to give the

\begin{proof}[Proof of part (b) of Lemma \ref{mevzubu}]
Let $\epsilon:=1-B(0)$. Note that $\epsilon>0$ by (\ref{birdenkucuk}). Lemma \ref{kontrolediyoriki} implies that
$$\sum_{k=N+1}^\infty\sup_{z\in\mathbb{V}_d'}\sup_{\theta\in\mathcal{C}_a(\kappa_5)}\left(B_{k,z}(\theta)-B_{k,z}(0)\right)\leq\sum_{k=N+1}^\infty\sup_{z\in\mathbb{V}_d'}\sup_{\theta\in\mathcal{C}_a(\kappa_5)}B_{k,z}(\theta)<\epsilon/2$$
for some $\kappa_5\in(0,\kappa_3)$ and $N\geq1$. Also, for every $k\geq1$, there exists a $\kappa_4=\kappa_4(k,\epsilon/{2N})\in(0,\kappa_3)$ such that
$$\sup_{z\in\mathbb{V}_d'}\sup_{\theta\in\mathcal{C}_a(\kappa_4)}\left(B_{k,z}(\theta)-B_{k,z}(0)\right)<\epsilon/{2N}$$
by Lemma \ref{kontrolediyorbir}.

Let $\kappa_{eq}:=\min\left(\kappa_5,\kappa_4(1,\epsilon/{2N}),\kappa_4(2,\epsilon/{2N}),\ldots,\kappa_4(N,\epsilon/{2N})\right)$. Recall (\ref{kontrolet}). For every $\theta\in\mathcal{C}_a(\kappa_{eq})$,
\begin{align*}
B(\theta)&\leq B(0)+\sum_{k=1}^{\infty}\sup_{z\in\mathbb{V}_d'}\left(B_{k,z}(\theta)-B_{k,z}(0)\right)\\
&=1-\epsilon+\sum_{k=1}^{N}\sup_{z\in\mathbb{V}_d'}\left(B_{k,z}(\theta)-B_{k,z}(0)\right)\ +\sum_{k=N+1}^{\infty}\sup_{z\in\mathbb{V}_d'}\left(B_{k,z}(\theta)-B_{k,z}(0)\right)\\
&<1-\epsilon+N(\epsilon/{2N})+\epsilon/2=1.\qedhere
\end{align*}
\end{proof}

\subsection{Tail estimates for joint regenerations}\label{jointregsection}

Recall that our proof of Theorem \ref{QequalsA} given in Section \ref{turetmis} relies on Lemma \ref{L2lemma} which, in turn, is proved in Subsection \ref{genclik} assuming Lemma \ref{mevzubu}. In Subsection \ref{noninterSection}, the latter assumption is verified using yet another result, namely Lemma \ref{jointreglemma}. Therefore, in order to complete the proof of Theorem \ref{QequalsA}, we need to prove Lemma \ref{jointreglemma}. 

For every $i\geq1$, $\theta\in\mathcal{C}_a(\kappa_3)$ and $z\in\mathbb{V}_d$, it follows from H\"older's inequality that
\begin{align}
h_{i,z}(\theta)&=\frac{\mathrm{e}^{-\langle\theta,z\rangle}}{P_{o,z}(l_1=0)}E_{o,z}\left[f(\theta,i,X,\tilde{X}),i=\inf\{l\in\mathcal{L}:l>0, X_{H_l}\neq\tilde{X}_{\tilde{H}_l}\},i>\nu_1, \beta=\infty, \tilde\beta=\infty\right]\nonumber\\
&\leq E_{o,z}\left[\exp\left\{4\langle\theta,X_{H_i}\rangle-4\Lambda_a(\theta)H_i\right\},\beta=\infty\right]^{1/4}\mathrm{e}^{-\langle\theta,z\rangle}E_{o,z}\left[\exp\{4\langle\theta,\tilde{X}_{\tilde{H}_i}\rangle-4\Lambda_a(\theta)\tilde{H}_i\},\tilde\beta=\infty\right]^{1/4}\nonumber\\
&\quad\times P_{o,z}(l_1=0)^{-1}P_{o,z}\left(i=\inf\{l\in\mathcal{L}:l>0, X_{H_l}\neq\tilde{X}_{\tilde{H}_l}\}, l_1=0\right)^{1/4}P_{o,z}\left(i>\nu_1\right)^{1/4}\nonumber\\
&= E_o\left[\exp\left\{4\langle\theta,X_{H_i}\rangle-4\Lambda_a(\theta)H_i\right\},\beta=\infty\right]^{1/2}\label{karatoprak}\\
&\quad\times P_{o,z}(l_1=0)^{-1}P_{o,z}\left(i=\inf\{l\in\mathcal{L}:l>0, X_{H_l}\neq\tilde{X}_{\tilde{H}_l}\}, l_1=0\right)^{1/4}P_{o,z}\left(i>\nu_1\right)^{1/4}.\nonumber
\end{align}
The next four lemmas control the terms in (\ref{karatoprak}).

\begin{lemma}\label{isimverbir}
There exists an $a_1<\infty$ such that
$$E_o\left[\left.\exp\left\{4\langle\theta,X_{H_i}\rangle-4\Lambda_a(\theta)H_i\right\}\right|\beta=\infty\right]\leq i\mathrm{e}^{a_1|\theta|i}$$
for every $i\geq1$ and $\theta\in\mathcal{C}_a(\kappa_3/4)$.
\end{lemma}

\begin{proof}
Recall $\kappa_3$ from Corollary \ref{mazaltov}.\\

\noindent\textit{(a) The non-nestling case:} For every $i\geq1$ and $\theta\in\mathcal{C}_a(\kappa_3/4)$,
\begin{align*}
&E_o\left[\left.\exp\left\{4\langle\theta,X_{H_i}\rangle-4\Lambda_a(\theta)H_i\right\}\right|\beta=\infty\right]\leq E_o\left[\left.\exp\left\{(4|\theta|+4|\Lambda_a(\theta)|)H_i\right\}\right|\beta=\infty\right]\\
&\qquad\leq E_o\left[\left.\exp\left\{8|\theta|\tau_i\right\}\right|\beta=\infty\right]=E_o\left[\left.\exp\left\{8|\theta|\tau_1\right\}\right|\beta=\infty\right]^i\leq E_o\left[\left.\exp\left\{2\kappa_3\tau_1\right\}\right|\beta=\infty\right]^{4|\theta|i/\kappa_3}
\end{align*}
by Jensen's inequality. Since $a_1:=\log E_o\left[\left.\exp\left\{2\kappa_3\tau_1\right\}\right|\beta=\infty\right]^{4/\kappa_3}<\infty$ by Corollary \ref{mazaltov}, we are done.

\vspace{0.4cm}
\noindent\textit{(b) The nestling case:} For every $i\geq1$ and $\theta\in\mathcal{C}_a(\kappa_3/4)$,
\begin{align*}
&E_o\left[\left.\exp\left\{4\langle\theta,X_{H_i}\rangle-4\Lambda_a(\theta)H_i\right\}\right|\beta=\infty\right]\leq E_o\left[\left.\exp\left\{4|\theta|\left|X_{H_i}\right|\right\}\right|\beta=\infty\right]\\
&\qquad=\sum_{j=0}^{i-1}\sum_{k=j}^{i-1}E_o\left[\left.\exp\left\{4|\theta|\left|X_{H_i}\right|\right\},\tau_j<H_i\leq\tau_{j+1},\langle X_{\tau_j},e_1\rangle=k\right|\beta=\infty\right]\\
&\qquad\leq\sum_{j=0}^{i-1}\sum_{k=j}^{i-1}E_o\left[\left.\exp\left\{4|\theta|\left|X_{\tau_j}\right|\right\},\tau_j=H_k\right|\beta=\infty\right]E_o\left[\left.\exp\left\{4|\theta|\left|X_{H_{i-k}}\right|\right\},H_{i-k}\leq\tau_1\right|\beta=\infty\right]\\
&\qquad\leq\sum_{j=0}^{i-1}E_o\left[\left.\exp\left\{4|\theta|\left|X_{\tau_j}\right|\right\}\right|\beta=\infty\right]E_o\left[\left.\sup_{1\leq n\leq\tau_1}\exp\left\{4|\theta|\left|X_n\right|\right\}\right|\beta=\infty\right]\\
&\qquad\leq\sum_{j=0}^{i-1}E_o\left[\left.\sup_{1\leq n\leq\tau_1}\exp\left\{4|\theta|\left|X_n\right|\right\}\right|\beta=\infty\right]^{j+1}\\
&\qquad\leq iE_o\left[\left.\sup_{1\leq n\leq\tau_1}\exp\left\{4|\theta|\left|X_n\right|\right\}\right|\beta=\infty\right]^i\leq iE_o\left[\left.\sup_{1\leq n\leq\tau_1}\exp\left\{\kappa_3\left|X_n\right|\right\}\right|\beta=\infty\right]^{4|\theta|i/\kappa_3}.
\end{align*}
Since $a_1:=\log E_o\left[\left.\sup_{1\leq n\leq\tau_1}\exp\left\{\kappa_3\left|X_n\right|\right\}\right|\beta=\infty\right]^{4/\kappa_3}<\infty$ by Corollary \ref{mazaltov}, we are done.
\end{proof}

\begin{lemma}\label{isimveriki}
For every $p\geq1$, there exists an $A_1=A_1(p)<\infty$ such that $$P_{o,z}\left(i>\nu_1\right)\leq A_1i^p\max(1,|z|)^{-p}$$ for every $i\geq1$ and $z\in\mathbb{V}_d$.
\end{lemma}

\begin{proof}
For every $i\geq1$, $z\in\mathbb{V}_d'$ and $p\geq1$,
\begin{align*}
P_{o,z}\left(i>\nu_1\right)&\leq P_{o,z}\left(\left\{X_n:0\leq n\leq\tau_i\right\}\cap\left\{\tilde{X}_m:0\leq m\leq\tilde{\tau}_i\right\}\neq\emptyset\right)\\
&\leq P_{o,z}\left(\tau_i\geq\frac{|z|}{2}\right) + P_{o,z}\left(\tilde{\tau}_i\geq\frac{|z|}{2}\right)= 2P_o\left(\tau_i\geq\frac{|z|}{2}\right)\\
&\leq 2\left(\frac{|z|}{2}\right)^{-p}E_o\left[\tau_i^p\right] = 2^{p+1}|z|^{-p}E_o\left[\{\tau_1+\cdots+(\tau_i-\tau_{i-1})\}^p\right]\\
&\leq 2^{p+1}|z|^{-p}i^{p-1}E_o\left[\tau_1^p+\cdots+(\tau_i-\tau_{i-1})^p\right]\\
&= 2^{p+1}|z|^{-p}i^{p-1}\left(E_o\left[\tau_1^p\right] + (i-1)E_o\left[\left.\tau_1^p\right|\beta=\infty\right]\right)\\
&\leq 2^{p+1}P_o(\beta=\infty)^{-1}E_o\left[\tau_1^p\right]i^p|z|^{-p}
\end{align*}
by H\"older's inequality.
Since $A_1(p):=2^{p+1}P_o(\beta=\infty)^{-1}E_o\left[\tau_1^p\right]<\infty$ by Corollary \ref{mazaltov}, we are done.
\end{proof}

\begin{lemma}\label{guloruc}
$\sup_{z\in\mathbb{V}_d}E_{o,z}\left[\left.\mathrm{e}^{a_3l^+}\right|l_1=0\right]<\infty$ for some $a_3>0$, where $l^+:=\inf\{l\in\mathcal{L}:l>0\}$. 
\end{lemma}

\begin{proof}
For every nearest-neighbor path $\left(x_i\right)_{i\geq0}$ on $\mathbb{Z}^d$, define
$$\beta'(\left(x_i\right)_{i\geq0}):=\inf\{i\geq1:\langle x_i,e_1\rangle<\langle x_o,e_1\rangle\}\quad\mbox{and}\quad M(\left(x_i\right)_{i\geq0}):=\sup\{\langle x_i,e_1\rangle: 0\leq i<\beta'(\left(x_i\right)_{i\geq0})\}.$$
In particular, $\beta'(X)=\beta$ and $\beta'(\tilde{X})=\tilde\beta$ for $X=\left(X_n\right)_{n\geq0}$ and $\tilde{X}=(\tilde{X}_m)_{m\geq0}$. Define $$\lambda=\lambda(X,\tilde X):=\left(M(X)\wedge M(\tilde X)\right)+1,\quad\lambda_1:=1\quad\mbox{and}\quad\lambda_{j+1}:=\lambda\left((X_n)_{n\geq H_{\lambda_j}},(\tilde{X}_m)_{m\geq\tilde{H}_{\lambda_j}}\right)$$ for every $j\geq1$. It is easy to see that $l^+:=\sup\{\lambda_j:\lambda_j<\infty\}$ when $X_o\in\mathbb{V}_d$ and $\tilde{X}_o\in\mathbb{V}_d$.

For every $z\in\mathbb{V}_d$,
\begin{align*}
E_{o,z}\left[\mathrm{e}^{\kappa_3\lambda},\lambda<\infty\right]&\leq E_{o,z}\left[\mathrm{e}^{\kappa_3(M(X)+1)},\beta<\infty\right] + E_{o,z}\left[\mathrm{e}^{\kappa_3(M(\tilde X)+1)},\tilde\beta<\infty\right]\\
&= 2E_o\left[\mathrm{e}^{\kappa_3(M(X)+1)},\beta<\infty\right]\leq 2E_o\left[\sup_{1\leq n\leq\tau_1}\mathrm{e}^{\kappa_3|X_n|},\beta<\infty\right].
\end{align*}
By H\"older's inequality,
\begin{align*}
E_{o,z}\left[\mathrm{e}^{a\lambda},\lambda<\infty\right]&\leq E_{o,z}\left[\mathrm{e}^{\kappa_3\lambda},\lambda<\infty\right]^{a/\kappa_3}P_{o,z}(\lambda<\infty)^{1-a/\kappa_3}\\
&\leq\left(2E_o\left[\sup_{1\leq n\leq\tau_1}\mathrm{e}^{\kappa_3|X_n|}\right]\right)^{a/\kappa_3}\Big(1-P_{o,z}(l_1=0)\Big)^{1-a/\kappa_3}
\end{align*} for every $a\in(0,\kappa_3)$. Therefore, it follows from Corollary \ref{mazaltov} and Lemma \ref{nonintersectlemma} that
\begin{equation}\label{lambada}
\sup_{z\in\mathbb{V}_d}E_{o,z}\left[\mathrm{e}^{a_3\lambda},\lambda<\infty\right]\leq\left(2E_o\left[\sup_{1\leq n\leq\tau_1}\mathrm{e}^{\kappa_3|X_n|}\right]\right)^{a_3/\kappa_3}\left(1-\inf_{z\in\mathbb{V}_d}P_{o,z}(l_1=0)\right)^{1-a_3/\kappa_3}\!\!\!\!<1
\end{equation} for some $a_3\in(0,\kappa_3)$.

For every $j\geq2$, 
\begin{align*}
E_{o,z}\left[\mathrm{e}^{a_3\lambda_j},\lambda_j<\infty\right]&=\sum_{z'\in\mathbb{V}_d}E_{o,z}\left[\mathrm{e}^{a_3\lambda_j},\tilde{X}_{\tilde{H}_{\lambda_{j-1}}}\!\!\!-X_{H_{\lambda_{j-1}}}=z',\lambda_j<\infty\right]\\
&=\sum_{z'\in\mathbb{V}_d}E_{o,z}\left[\mathrm{e}^{a_3\lambda_{j-1}},\lambda_{j-1}<\infty,\tilde{X}_{\tilde{H}_{\lambda_{j-1}}}\!\!\!-X_{H_{\lambda_{j-1}}}=z'\right]E_{o,z'}\left[\mathrm{e}^{a_3\lambda},\lambda<\infty\right]\\
&\leq E_{o,z}\left[\mathrm{e}^{a_3\lambda_{j-1}},\lambda_{j-1}<\infty\right]\sup_{z'\in\mathbb{V}_d}E_{o,z'}\left[\mathrm{e}^{a_3\lambda},\lambda<\infty\right].
\end{align*}
By induction,
$$\sup_{z\in\mathbb{V}_d}E_{o,z}\left[\mathrm{e}^{a_3\lambda_j},\lambda_j<\infty\right]\leq\mathrm{e}^{a_3}\left(\sup_{z\in\mathbb{V}_d}E_{o,z}\left[\mathrm{e}^{a_3\lambda},\lambda<\infty\right]\right)^{j-1}.$$
Therefore,
\begin{align*}
\sup_{z\in\mathbb{V}_d}E_{o,z}\left[\mathrm{e}^{a_3l^+}\right]&\leq\sum_{j=1}^\infty\sup_{z\in\mathbb{V}_d}E_{o,z}\left[\mathrm{e}^{a_3l^+},l^+=\lambda_j\right]\leq\sum_{j=1}^\infty\sup_{z\in\mathbb{V}_d}E_{o,z}\left[\mathrm{e}^{a_3\lambda_j},\lambda_j<\infty\right]\\
&\leq\sum_{j=1}^\infty\mathrm{e}^{a_3}\left(\sup_{z\in\mathbb{V}_d}E_{o,z}\left[\mathrm{e}^{a_3\lambda},\lambda<\infty\right]\right)^{j-1}=\mathrm{e}^{a_3}\left(1-\sup_{z\in\mathbb{V}_d}E_{o,z}\left[\mathrm{e}^{a_3\lambda},\lambda<\infty\right]\right)^{-1}<\infty
\end{align*}
by (\ref{lambada}). This implies the desired result since $\inf_{z\in\mathbb{V}_d}P_{o,z}(l_1=0)>0$ by Lemma \ref{nonintersectlemma}.
\end{proof}

\begin{lemma}\label{gulnamaz}
There exist $a_2>0$ and $A_2<\infty$ such that $P_{o,z}\left(\left.l^*=i\right|l_1=0\right)\leq A_2\mathrm{e}^{-a_2i}$
for every $i\geq1$ and $z\in\mathbb{V}_d$, where $l^*:=\inf\{l\in\mathcal{L}:l>0, X_{H_l}\neq\tilde{X}_{\tilde{H}_l}\}$. 
\end{lemma}

\begin{proof}
Fix $a_3$ as in Lemma \ref{guloruc}. Define $\nu_1^+:=\inf\{m>0:X_i=\tilde{X}_j\mbox{ for some }i\geq0,j\geq0,\mbox{ and }\langle X_i,e_1\rangle=m\}$. For every $z\in\mathbb{V}_d$ and $a\in(0,a_3)$, by H\"older's inequality,
\begin{align*}
E_{o,z}\left[\left.\mathrm{e}^{al^+},\nu_1^+<\infty\right|l_1=0\right]&\leq E_{o,z}\left[\left.\mathrm{e}^{a_3l^+}\right|l_1=0\right]^{a/a_3}P_{o,z}\left(\left.\nu_1^+<\infty\right|l_1=0\right)^{1-a/a_3}\quad\mbox{and}\\
\sup_{z\in\mathbb{V}_d}E_{o,z}\left[\left.\mathrm{e}^{al^+},\nu_1^+<\infty\right|l_1=0\right]&\leq\left(\sup_{z\in\mathbb{V}_d}E_{o,z}\left[\left.\mathrm{e}^{a_3l^+}\right|l_1=0\right]\right)^{a/a_3}\left(\sup_{z\in\mathbb{V}_d}P_{o,z}\left(\left.\nu_1^+<\infty\right|l_1=0\right)\right)^{1-a/a_3}.
\end{align*} It is an easy consequence of (\ref{ellipticity}) and Lemma \ref{nonintersectlemma} that $\sup_{z\in\mathbb{V}_d}P_{o,z}\left(\left.\nu_1^+<\infty\right|l_1=0\right)<1$. Hence,
$$\sup_{z\in\mathbb{V}_d}E_{o,z}\left[\left.\mathrm{e}^{a_2l^+},\nu_1^+<\infty\right|l_1=0\right]<1$$ for some $a_2\in(0,a_3)$.

It follows from the regeneration structure that
\begin{align*}
E_{o,z}\left[\left.\mathrm{e}^{a_2l^*}\right|l_1=0\right]&=E_{o,z}\left[\left.\mathrm{e}^{a_2l^*},l^+<l^*\right|l_1=0\right] + E_{o,z}\left[\left.\mathrm{e}^{a_2l^*},l^+=l^*\right|l_1=0\right]\\
&=E_{o,z}\left[\left.\mathrm{e}^{a_2l^+},X_{H_{l^+}}=\tilde{X}_{\tilde{H}_{l^+}}\right|l_1=0\right]E_{o,o}\left[\left.\mathrm{e}^{a_2l^*}\right|l_1=0\right] + E_{o,z}\left[\left.\mathrm{e}^{a_2l^+},l^+=l^*\right|l_1=0\right]\\
&\leq E_{o,z}\left[\left.\mathrm{e}^{a_2l^+},\nu_1^+<\infty\right|l_1=0\right]E_{o,o}\left[\left.\mathrm{e}^{a_2l^*}\right|l_1=0\right] + E_{o,z}\left[\left.\mathrm{e}^{a_2l^+}\right|l_1=0\right].
\end{align*}
Therefore,
$$\sup_{z\in\mathbb{V}_d}E_{o,z}\!\left[\left.\mathrm{e}^{a_2l^*}\right|l_1=0\right]\leq\sup_{z\in\mathbb{V}_d}E_{o,z}\!\left[\left.\mathrm{e}^{a_2l^+},\nu_1^+<\infty\right|l_1=0\right]\sup_{z\in\mathbb{V}_d}E_{o,z}\!\left[\left.\mathrm{e}^{a_2l^*}\right|l_1=0\right] + \sup_{z\in\mathbb{V}_d}E_{o,z}\!\left[\left.\mathrm{e}^{a_2l^+}\right|l_1=0\right]$$
and
$$A_2:=\sup_{z\in\mathbb{V}_d}E_{o,z}\left[\left.\mathrm{e}^{a_2l^*}\right|l_1=0\right]\leq\left(1-\sup_{z\in\mathbb{V}_d}E_{o,z}\left[\left.\mathrm{e}^{a_2l^+},\nu_1^+<\infty\right|l_1=0\right]\right)^{-1}\sup_{z\in\mathbb{V}_d}E_{o,z}\left[\left.\mathrm{e}^{a_2l^+}\right|l_1=0\right]<\infty.$$
Finally, for every $i\geq1$ and $z\in\mathbb{V}_d$,
$P_{o,z}\left(\left.l^*=i\right|l_1=0\right)\leq E_{o,z}\left[\left.\mathrm{e}^{a_2l^*}\right|l_1=0\right]\mathrm{e}^{-a_2i}\leq A_2\mathrm{e}^{-a_2i}$
by Chebyshev's inequality.
\end{proof}

\begin{proof}[Proof of Lemma \ref{jointreglemma}]
Recall (\ref{karatoprak}). For every $\kappa_5\in(0,\kappa_3/4)$ and $p\geq1$, Lemmas \ref{isimverbir}, \ref{isimveriki} and \ref{gulnamaz} imply that
\begin{align}
\sum_{i=1}^\infty\sum_{z\in\mathbb{V}_d}\sup_{\theta\in\mathcal{C}_a(\kappa_5)}h_{i,z}(\theta)&\leq\sum_{i=1}^\infty\sum_{z\in\mathbb{V}_d}\sup_{\theta\in\mathcal{C}_a(\kappa_5)}E_o\left[\exp\left\{4\langle\theta,X_{H_i}\rangle-4\Lambda_a(\theta)H_i\right\},\beta=\infty\right]^{1/2}P_{o,z}(l_1=0)^{-1}\nonumber\\
&\qquad\qquad\quad\times P_{o,z}\left(i=\inf\{l\in\mathcal{L}:l>0, X_{H_l}\neq\tilde{X}_{\tilde{H}_l}\}, l_1=0\right)^{1/4}P_{o,z}\left(i>\nu_1\right)^{1/4}\nonumber\\
&\leq\sum_{i=1}^\infty\sum_{z\in\mathbb{V}_d}\sup_{\theta\in\mathcal{C}_a(\kappa_5)}\left(i\mathrm{e}^{a_1|\theta|i}\right)^{1/2}P_{o,z}(l_1=0)^{-1}\left(A_2\mathrm{e}^{-a_2i}\right)^{1/4}\left(A_1i^p\max(1,|z|)^{-p}\right)^{1/4}\nonumber\\
&\leq\frac{(A_1A_2)^{1/4}}{\inf_{z\in\mathbb{V}_d}P_{o,z}(l_1=0)}\left(\sum_{i=1}^\infty i^{1/2+p/4}\exp\{(2a_1\kappa_5-a_2)i/4\}\right)\left(1+\sum_{z\in\mathbb{V}_d'}|z|^{-p/4}\right)\label{alplere}
\end{align}
for some $a_1<\infty$, $a_2>0$, $A_1<\infty$ and $A_2<\infty$. Clearly, (\ref{alplere}) is finite when $p>4d$ and $\kappa_5\in(0,a_2/2a_1)$.
\end{proof}

\section*{Acknowledgments}

I sincerely thank O.\ Zeitouni for many valuable discussions and comments.

\bibliographystyle{plain}
\bibliography{references}
\end{document}